\documentclass[a4paper]{amsart}

\usepackage{amsmath,amssymb,amsbsy}
  \usepackage{paralist}
  \usepackage{graphics} %% add this and next lines if pictures should be in esp format
  \usepackage{epsfig} %For pictures: screened artwork should be set up with an 85 or 100 line screen
\usepackage{amsthm}
\usepackage{tikz}
\usepackage[all]{xy}
 \usepackage{verbatim}
\usepackage[bookmarks]{hyperref}

\newtheorem{theorem}{Theorem}[section]
\newtheorem{thmx}{Theorem}

\newtheorem{proposition}[theorem]{Proposition}
\newtheorem{corollary}[theorem]{Corollary}
\newtheorem{lemma}[theorem]{Lemma}
\newtheorem{main lemma}[theorem]{Main Lemma}

{\theoremstyle{definition}
\newtheorem{definition}{Definition}

\newtheorem{remark}{Remark}

\newcommand{\C}{\mathbb{C}}
\newcommand{\R}{\mathbb{R}}

\newcommand{\N}{\mathbb{N}}
\def\S{\mathbb{S}}
\newcommand{\mf}[1]{\mathbf{#1}}

\def\pa{\partial}

\def\eps{\varepsilon}
\def\wc{\rightharpoonup}
\newcommand{\de}[1] {\mathrm{d} #1}

\DeclareMathOperator{\loc}{loc}
\DeclareMathOperator{\supp}{supp}

\DeclareMathOperator{\Int}{Int}

\title[Liouville theorems and $1$-dimensional symmetry]{Liouville theorems and $1$-dimensional symmetry for solutions of an elliptic system modelling phase separation}
\author{Nicola Soave and Susanna Terracini}

\address{
\hbox{\parbox{5.7in}{\medskip\noindent
 Nicola Soave\\
Mathematisches Institut, \\
Justus-Liebig-Universit\"at Giessen, \\
Arndtstrasse 2, 35392 Giessen (Germany),\\[2pt]
{\em{E-mail address: }}{\tt nicola.soave@gmail.com, nicola.soave@math.uni-giessen.de.}\\[5pt]
 S. Terracini\\
 Dipartimento di Matematica ``Giuseppe Peano'',\\
Universit\`a di Torino, \\
Via Carlo Alberto, 10,
10123 Torino (Italy). \\[2pt]
                                     \em{E-mail address: }{\tt susanna.terracini@unito.it.}}}
}

\thanks{
\indent The authors are partially supported through the project ERC Advanced Grant 2013 n. 339958 ``Complex Patterns for Strongly Interacting Dynamical Systems - COMPAT". 
 }

\begin{document}

%%%%%%%%%%%%%%%%%%%%%%%%%%%%%%%%%%%%%%%%%%%%%%%%%%%%%%%%%%%%%%%%%%%%%%%%%%%%%%%

%\begin{document}
%
%\begin{center}
%\Large{\sc{Liouville-type theorems and $1$-dimensional symmetry for a competitive system modelling phase separation}}
%\end{center}
%
%
%\footnotesize
%{\centerline{Nicola Soave}
% \centerline{Justus Liebig Universit\"at Giessen}
%   \centerline{address}
%%\normalsize
%   \centerline{and}
%   \centerline{Susanna Terracini}
%%\footnotesize
% \centerline{Universit\`a degli Studi di Torino}
%   \centerline{address}
%}
\begin{abstract}
We consider solutions of the competitive elliptic system
\begin{equation}\tag{S}\label{syst 1}
\begin{cases}
-\Delta u_i = - \sum_{j \neq i} u_i u_j^2 & \text{in $\R^N$} \\
u_i >0 & \text{in $\R^N$} 
\end{cases} \qquad i=1,\dots,k.
\end{equation}
We are concerned with the classification of entire solutions, according with their growth rate. The prototype of our main results is the following: there exists a function $\delta=\delta(k,N) \in \N$, increasing in $k$, such that if $(u_1,\dots,u_k)$ is a solution of \eqref{syst 1} and
\[
u_1(x)+\cdots+u_k(x) \le C(1+|x|^d) \qquad \text{for every $x \in \R^N$},
\]
then $d \ge \delta$. This means that the number of components $k$ of the solution imposes a lower bound, increasing in $k$, on the minimal growth of the solution itself. If $N=2$, the expression of $\delta$ is explicit and optimal, while in higher dimension it can be characterized in terms of an optimal partition problem. We discuss the sharpness of our results and, as a further step, for every $N \ge 2$ we can prove the $1$-dimensional symmetry of the solutions of \eqref{syst 1} satisfying suitable assumptions, extending known results which are available for $k=2$. The proofs rest upon a blow-down analysis and on some monotonicity formulae. 
\end{abstract}

\maketitle

\section{Introduction}

This paper concerns the classification of \emph{positive} or \emph{nonnegative} entire solutions with algebraic growth of the competitive elliptic system
\begin{equation}\label{system}
 -\Delta u_i=-\sum_{\substack{j=1 \\ j \neq i}}^k u_j^2 u_i  \qquad \text{in $\R^N$, for $i=1,\dots,k$},
\end{equation}
with $k \ge 2$. Here and in the what follow, writing \emph{positive solution} we mean that $u_i>0$ in $\R^N$ for every $i$, while writing \emph{nonnegative solution} we admit the possibility that some $u_i$ vanish identically, requiring however that at least two components are non-trivial. Note that, by the strong maximum principle, if $u_i \ge 0$ and $u_i \not \equiv 0$, then $u_i>0$ in $\R^N$. The main result we aim at proving is the following Liouville-type theorem.
\begin{theorem}\label{protothm 1}
Let $N \ge 2$, and let $(u_1,\dots,u_k)$ be a positive solution of \eqref{system} having algebraic growth, that is, there exists $C,d>0$ such that
\begin{equation}\label{alg growth}
u_1(x)+\cdots+u_k(x) \le C(1+ |x|^d) \qquad \text{for every $x \in \R^N$}.
\end{equation}
Then 
\[
d \ge \sqrt{\left(\frac{N-2}{2}\right)^2+\mathcal{L}_k(\S^{N-1})}-\frac{N-2}{2},
\]
where $\mathcal{L}_k(\S^{N-1})$ is the spectral minimal partition sequence of $-\Delta_{\S^{N-1}}$ over $\S^{N-1}$, introduced in \cite{HeHHTer1} (see definition \eqref{def di L_k}). Note that  $\mathcal{L}_k(\S^{1})={k^2}/{4}$ when $N=2$.
\end{theorem}
In one direction, this means that if we consider a positive solution of \eqref{system} with a prescribed number $k$ of components, then we have a minimal admissible growth for the solution itself. As we will prove in Lemma \ref{lem: monot part}, the minimal growth is strictly increasing in $k$. In the opposite direction, we deduce that a bound on the growth of a positive solution imposes a bound on the number of components $k$ of the solution itself. 
% If $N=2$, we derive an explicit and optimal expression of $h(d,2)$, while in higher dimension we characterize it in terms of an optimal partition problem already introduced in \cite{HeHHTer1}.
When $N \ge 3$, the exact value of $\mathcal{L}_k(\S^{N-1})$ is known for $k=2$, and we will be able to solve it in the special case of $k=3$ components, thus extending what has been proved in \cite{HeHHTer2} for the two-sphere. As a consequence, we will prove the following. 

\begin{corollary}\label{protothm 22}
Let $N ,k \ge 2$, and let $(u_1,\dots,u_k)$ be a positive solution of \eqref{system}.\\
($i$) If the solution has linear growth, that is there exists $C>0$ such that
\begin{equation}\label{linear growth}
u_1(x)+\cdots+u_k(x) \le C(1+ |x|) \qquad \text{for every $x \in \R^N$},
\end{equation}
then $k = 2$ and the solution has growth rate $1$. \\
%and the solution has exactly linear growth, that is $u_1(x)+u_2(x) \simeq |x|$ as $|x| \to +\infty$.\\
($ii$) If there exists $C>0$ such that 
%\begin{equation}\label{ass 3/2}
\[
u_1(x)+\cdots+u_k(x) \le C(1+|x|^{3/2}) \qquad \text{for every $x \in \R^N$}.
\]
%\end{equation}
Then either $k=2$ and the solution has linear growth, 
or $k=3$ and the solution has growth rate $3/2$.
\end{corollary}

Here and in the rest of the paper we write that $(u_1,\dots,u_k)$ \emph{has growth rate $d$} if
\[
\lim_{r \to +\infty} \frac{\frac{1}{r^{N-1}} \int_{\pa B_r} \sum_{i=1}^k u_i^2  }{r^{2d'}} = \begin{cases} +\infty & \text{if $d'<d$} \\
0 & \text{if $d'>d$},
\end{cases}
\]
where $B_r$ denotes the ball of center $0$ and radius $r$. We will observe that any solution of \eqref{system} has a growth rate.

As a further step, we address the proof of the validity of some De Giorgi-type conjectures for solutions of \eqref{system}: under suitable assumptions, we show that a solution of \eqref{system} is necessarily $1$-dimensional, namely up to a rotation it depends only on $1$ variable. In what follows we write that $(u_1,\dots,u_k)$ \emph{has algebraic growth} if it satisfies condition \eqref{alg growth} for some $C>0$ and $d>1$. If the stronger condition \eqref{linear growth} holds, we write that $(u_1,\dots,u_k)$ \emph{has linear growth}.
\begin{theorem}\label{thm: 1-dimensional symmetry}
Let $N \ge 2$, let $(u_1,\dots,u_k)$ be a nonnegative solution of \eqref{system}.
\begin{itemize}
\item[($i$)] If $(u_1,\dots,u_k)$ has linear growth, then all the components but two, say $u_1$ and $u_2$, are identically zero, and $(u_1,u_2)$ is $1$-dimensional. 
\item[($ii$)] If $(u_1,\dots,u_k)$ has algebraic growth and for some $i \neq j$
\[
\lim_{x_N \to \pm \infty} \left(u_i(x',x_N) - u_j(x',x_N) \right) = \pm \infty, 
\]
the limits being uniform in $x' \in \R^{N-1}$, then $(u_1,\dots,u_k)$ has linear growth, all the components $u_l$ with $l \neq i,j$ are identically zero, and $(u_i,u_j)$ is $1$-dimensional.
%\item[($iii$)] If the dimension $N=2$, and for some $i \neq j$
%\[
%\pa_N u_i >0 \quad \text{and} \quad \pa_N u_j<0 \quad \text{in $\R^N$},
%\]
%then $d=1$, all the components $u_k$ with $k \neq i,j$ are identically zero, and $(u_i,u_j)$ is a $1$-dimensional solution of \eqref{system 2 comp}.
\end{itemize}
\end{theorem}
We postpone a more precise discussion of our main results after a brief review of what is known on existence and qualitative properties of solutions of \eqref{system}. Such review has to be understood also as a motivation for our study.

The $2$ components system 
\begin{equation}\label{system 2 comp}
 \begin{cases}
-\Delta u=- u v^2  & \text{in $\R^N$} \\
-\Delta v = - u^2 v & \text{in $\R^N$} \\
u,v>0 & \text{in $\R^N$}
 \end{cases}
\end{equation}
has been widely investigated in recent years. It appears in the analysis of phase separation phenomena in a binary mixture of Bose-Einstein condensates with multiple states; we refer to the papers \cite{BeLiWeZh} by H. Berestycki, T.-C. Lin, J. Wei and C. Zhao, \cite{BeTeWaWe}  by H. Berestycki, K. Wang, J. Wei and the second author, and to the references therein for a detailed derivation of the phase separation model. 
%, and several contributions concerning existence and qualitative properties of solutions are available. 
In the quoted papers, the $1$-dimensional case has been completely classified: up to translations, scaling and exchange of the components there is only one positive solution $(u,v)$, which has linear growth, has the symmetry property $u(x)=v(-x)$ for every $x \in \R$, and satisfies the monotonicity condition $u'>0$, $v'<0$ in $\R$. The linear growth is the least admissible growth rate for positive solutions to \eqref{system 2 comp}; indeed in any dimension $N \ge 1$, if $(u,v)$ is a nonnegative solution of \eqref{system 2 comp} and satisfies the sublinear growth condition
\[
u(x)+v(x) \le C(1+|x|^\alpha) \qquad \text{in $\R^N$}
\]
for some $\alpha \in (0,1)$ and $C>0$, then one between $u$ and $v$ is $0$, and the other has to be constant. This has been proved by B. Noris, H. Tavares, G. Verzini and the second author in \cite{NoTaTeVe}, see Proposition 2.6, and together with its counterpart for systems with $k$ components, Proposition 2.7 in the same paper, is the only known example of Liouville-type theorem available for system \eqref{system}.

The non-existence of positive solutions having sublinear growth, and the existence of a positive solution with linear growth, suggest an analogy between problem \eqref{system 2 comp} and the Laplace equation. This point is made clear in \cite{BeTeWaWe}, where for every integer $d \in \N$ the authors constructed a positive solution $(u_d,v_d)$ of \eqref{system 2 comp} ``modelled on" the homogeneous harmonic polynomial $\Psi_d = \mathfrak{Re}(z^d)$, in the sense that $(u_d,v_d)$ has growth rate $d$ (the same asymptotic growth of $\Psi_d$), and $(u_d,v_d)$ exhibits the symmetry of $(\Psi_d^+,\Psi_d^-)$; in this way the authors associated to any homogeneous harmonic polynomial of two variables a positive solution of \eqref{system 2 comp}. Also the converse can be done: to any positive solution to \eqref{system 2 comp} having algebraic growth it is possible to associate a class of homogeneous harmonic polynomials, see the blow-down Theorem 1.4 in \cite{BeTeWaWe}. 
%In particular, to any solutions $(u,v)$ of \eqref{system} it is possible to associate a positive integer $d \in \N$ such that
%\[
%\lim_{r \to +\infty} \frac{1}{r^{N-1+2d'}} \int_{\pa B_r} u^2 +v^2 =\begin{cases} 0 & \text{if $d'<d$} \\
%+\infty & \text{if $d'>d$}
%\] 
%
%
% having algebraic growth is such that
%\[
%u(x)+v(x) \simeq |x|^{d} \qquad \text{as $|x| \to +\infty$}
%\]
%for some \emph{positive integer} $d \in \N$. 
It is worth to point out that the dichotomy ``positive solutions to \eqref{system 2 comp}" -- ``harmonic function" is not an exclusive prerogative of solutions having algebraic growth, as revealed by the existence of solutions with exponential growth which are associated to exponential harmonic functions, for which we refer to the main results in \cite{SoZi} by A. Zilio and the first author. 

Most of the quoted achievements admit a natural counterpart for the $k$ components system \eqref{system} with $k>2$. In particular, for any $k>2$ there exist infinitely many positive solutions having algebraic or exponential growth (see Theorem 1.6 in \cite{BeTeWaWe} and Theorem 1.8 in \cite{SoZi}), which are ``modelled on" suitable harmonic functions. For this reason, the reader could be tempted to think that the qualitative description of the $k$ components system is essentially the same than that of the $2$ component system. As we shall see, when $k>2$ the picture is more involved. In what follows we restrict our attention to solutions having algebraic growth and, in order to better motivate our study, we report two aforementioned results in \cite{BeTeWaWe}. Concerning the notation, here and in the rest of the paper we denote by $B_r(x_0)$ the ball of centre $x_0$ and radius $r$ in $\R^N$, and write simply $B_r$ for $B_r(0)$; we use the complex notation $z=x+iy$ for points of $\C \simeq \R^2$, writing $\bar z$ for the complex conjugate of $z$, and we count the indexes $i=1,\dots,k,k+1,\dots$ modulus $k$.

\begin{thmx}[Theorem 1.4 in \cite{BeTeWaWe}]\label{thm blow down 2 comp}
Let $N \ge 2$, $(u,v)$ be a positive solution of \eqref{system}, and let us introduce 
\[
(u_{R}(x),v_R(x)):=\left(\frac{1}{R^{N-1}}\int_{\pa B_R} u^2+v^2\right)^{-1/2} (u(Rx),v(Rx)).
\]
Let us assume that 
\begin{equation}\label{asymptotic growth}
\lim_{r \to +\infty} \frac{r \int_{B_r} |\nabla u|^2 +|\nabla v|^2 +u^2 v^2}{\int_{\pa B_r} u^2 +v^2} =: d < +\infty.
\end{equation}
Then $d$ is a positive integer. Moreover, there exist a subsequence of the blow-down family $\{(u_{R},v_{R}): R >0\}$, and a homogeneous harmonic polynomial $\Psi$ of degree $d$, such that $(u_{R},v_{R}) \to (\Psi^+,\Psi^-)$ as $R \to +\infty$ in $\mathcal{C}^0_{\loc}(\R^N)$ and in $H^1_{\loc}(\R^N)$. 
%Moreover,
%\[
%H(R) R^2 \, u_{R}^2 v_{R}^2 \to 0 \qquad \text{in $L^1_{loc}(\R^N)$}.
%\]
\end{thmx} 

\begin{thmx}[Theorem 1.6 in \cite{BeTeWaWe}]\label{thm: ex tante comp}
Let $k \ge 2$ and $d \in \N/2$ such that $2d=hk$ for some $h \in \N$; let $G_{\pi/d}$ denote the rotation of angle $\pi/d$, with order $2d$. There exists a positive solution of system \eqref{system} in $\R^2$ such that
\begin{align*}
\text{($i$)} &\quad u_i(z) = u_{i+1}(G_{\pi/d}z) \qquad \text{in $\C$, $i=1,\dots,k$} \\
\text{($ii$)} &\quad u_{k+i-1}(z) = u_{i}( \bar z) \qquad \text{in $\C$, $i=1,\dots,k$} \\
\text{($iii$)} &\quad \lim_{r \to +\infty} \frac{r \int_{B_r} \sum_{i=1}^k |\nabla u_i|^2 +\sum_{1 \le i<j\le k}u_i^2 u_j^2}{\int_{\pa B_r} \sum_{i=1}^k u_i^2} = d \\
\text{($iv$)} &\quad  \lim_{r \to +\infty} \frac{1}{r^{1+2d}} \int_{\pa B_r} \sum_{i=1}^k u_i^2=b
\end{align*}  
for some $b \in (0,+\infty)$.
%Furthermore,
%\[
%\lim_{r \to +\infty} \frac{r \int_{B_r} \sum_{i=1}^k |\nabla u_i|^2 \sum_{i<j}u_i^2 u_j^2}{\int_{\pa B_r} \sum_{i=1}^k} = d,
%\]
%and there exists $b \in (0,+\infty)$ such that
%\[
%\lim_{r \to +\infty} \frac{1}{r^{1+2d}} \int_{\pa B_r} \sum_{i=1}^k u_i^2=b.
%\] 
\end{thmx}

As previously stated, Theorem \ref{thm blow down 2 comp} allows us to associate to any positive solution $(u,v)$ of \eqref{system 2 comp} a homogeneous harmonic polynomial. In particular, this implies a quantization of the admissible growth rates at infinity, see the limit \eqref{asymptotic growth} and the forthcoming Proposition \ref{prop: sui growth rates}.
%$u(x)+v(x) \simeq |x|^d$ as $|x| \to +\infty$, where $d \in \N$. 
The very same quantization cannot be expected when $k>2$: indeed, Theorem \ref{thm: ex tante comp} provides solutions with half-integer asymptotic growth for every odd $k$, see point ($iv$). More important, the presence of more than $2$ components prevents the possibility that, if an asymptotic profile exists, has the simple structure $(\Psi^+,\Psi^-)$ for some homogeneous harmonic polynomial $\Psi$. In light of these remarks, an interesting problem is the description of the asymptotic profiles of the solutions of \eqref{system}. As a further question, we observe that Theorem \ref{thm: ex tante comp} ensures the existence of a positive solution $(u,v)$ to \eqref{system} with minimal growth rate $3/2$ when $k=3$, $2$ when $k=4$, $5/2$ when $k=5$, $\ldots$; we recall that writing ``positive solution" we mean that $u_i>0$ in $\R^N$ for every $i$. It is natural to wonder if these are really the minimal admissible growth rates or not. In the opposite direction, is it true that if a nonnegative solution of \eqref{system} has growth rate $d$, then there exists a maximal number of components depending on $d$ and on the dimension $N$ which cannot vanish identically? We recall that in this spirit the non-existence results for positive solutions having sublinear growth holds also when $k>2$, see Proposition 2.7 in \cite{NoTaTeVe}.

The aim of this paper is to answer the previous open problems and questions. Moreover, once that such topics are discussed, we will be able to extend some results of $1$-dimensional symmetry of solutions in the present setting. The proof of the validity of some De Giorgi's-type conjectures for positive solutions of \eqref{system 2 comp} has been object of an increasing attention in the last years. In dimension $N=2$, A. Farina proved that if $(u,v)$ has algebraic growth and $\pa_2 u>0$ in $\R^2$, then $(u,v)$ is $1$-dimensional. This enhances a previous result in \cite{BeLiWeZh}, where the $1$-dimensional symmetry of $(u,v)$ was obtained under the linear growth assumption of $(u,v)$ plus the monotonicity condition $\pa_2 u>0$ and $\pa_2 v<0$ in $\R^2$. Always in dimension $N=2$, in \cite{BeTeWaWe} it has been proved that if $(u,v)$ is a stable solution of \eqref{system 2 comp} having linear growth, then it is $1$-dimensional. Symmetry results in dimension $N=2$ for systems having a more general form, under either monotonicity or stability assumptions, have been achieved by S. Dipierro \cite{Dip}. In the higher dimensional case $N \ge 2$, A. Farina and the first author proved in \cite{FaSo} that if $(u,v)$ has algebraic growth and 
\[
\lim_{x_N \to \pm \infty} \left( u(x',x_N)-v(x',x_N) \right) = \pm \infty,
\] 
the limit being uniform in $x' \in \R^{N-1}$, then $(u,v)$ depends only on $x_N$. This positively answer a conjecture formulated in \cite{BeLiWeZh}. Furthermore, as product of the main results in \cite{Wa,Wa2}, K. Wang showed that if $(u,v)$ has linear growth (without other assumptions), then it is $1$-dimensional. 

As stated in Theorem \ref{thm: 1-dimensional symmetry}, our aim is to extend the two last quoted achievements for solutions of the $k$ components system \eqref{system}. 
%Concerning the $1$-dimensional symmetry of solutions satisfying suitable monotonicity conditions, one could be tempted to think that, in dimension $N=2$, if for some $i \neq j$
%\[
%\frac{\pa u_i}{\pa x_N}>0 \quad \text{and} \quad \frac{\pa u_j}{\pa x_N}<0 \quad \text{in $\R^N$},
%\]
%then $u_l \equiv 0$ for every $l \neq i,j$, and $(u_i,u_j)$ is a $1$-dimensional solution of \eqref{system 2 comp}. Actually this seems not be the case: the reader can think at the positive solution of the $4$ components system having quadratic growth. 

In what follows, we introduce convenient notations and state our main results in a precise form.

%, which we report here for the future convenience.
%
%\begin{theorem}[Theorem 1.4 in \cite{BeTeWaWe}]\label{thm blow down}
%Let $(u,v)$ be a solution of \eqref{system}, and assume that 
%\[
%\lim_{r \to +\infty} \frac{r \int_{B_r} |\nabla u|^2 +|\nabla v|^2 +u^2 v^2}{\int_{\pa B_r} u^2 +v^2} =: d < +\infty.
%\] 
%Then $d$ is a positive integer. Let 
%\[
%(u_{R}(x),v_R(x)):=\left(\frac{1}{R^{N-1}}\int_{\pa B_R} u^2+v^2\right)^{-1/2} (u(Rx),v(Rx)).
%\]
%There exist a subsequence of the blow down family $\{(u_{R},v_{R}): R >0\}$, and a homogeneous harmonic polynomial $\Psi$ of degree $d$, such that $(u_{R},v_{R}) \to (\Psi^+,\Psi^-)$ as $R \to +\infty$m in $\mathcal{C}^0_{loc}(\R^N)$ and in $H^1_{loc}(\R^N)$. Moreover,
%\[
%H(R) R^2 \, u_{R}^2 v_{R}^2 \to 0 \qquad \text{in $L^1_{loc}(\R^N)$}.
%\]
%\end{theorem} 

%On the other hand, for $N=2$ (thus for every $N \ge 2$) system \eqref{system 2 comp} admits infinitely many solutions having any integer order of algebraic growth at infinity, that is, $(u,v) \simeq |x|^d$ as $|x| \to +\infty$ 
%All the aforementioned existence results admits a

%We remark that, by the strong maximum principle, for every $i$ either $u_i>0$ in $\R^N$ or $u_i \equiv 0$. 

\subsection{Notation and further results.}

\begin{itemize}
\item We use the vector notation $\mf{u}:=(u_1,\dots,u_k)$.
\item Let $A_1,A_2$ be disjoint open subsets of $\R^N$; we write that $A_1$ and $A_2$ are \emph{adjacent} if $\pa A_i \cap \pa A_j$ has positive $(N-1)$-dimensional Hausdorff measure. 
%Let $A \subset \R^{N}$ be a cone. The solid angle of $A$ is defined as the surface integral $\int_{A} d \mathcal{H}^{N-1}$. \\
\item For any continuous function $u$ in $\R^N$, the set $\{u>0\}$ is called \emph{positivity domain} of $u$, and its connected components are called \emph{nodal domains}.
%%\item For any angle $\theta \in [0,2\pi)$, we denote by $G_\theta$ the rotation of angle $\theta$ in counterclockwise sense. 
\item For a vector valued function $\mf{u}$, we call \emph{nodal set} or \emph{zero level set} $\{\mf{u}=\mf{0}\}$.
\item For any $A \subset \R^N$, we write $\chi_A$ for the characteristic function of $A$.
\item For any $A \subset \R^N$, we write $\Int(A)$ for the interior of $A$.
\item In the proof of our results we often write u.t.s. instead of ``up to a subsequence".
\item The notation $\mathcal{H}^{m}(\Omega)$ is used for the $m$-dimensional Hausdorff measure of $\Omega \subset \R^N$.
\item For any $\omega \subset \pa B_1$, the first eigenvalue of the Laplace-Beltrami operator $-\Delta_{\S^{N-1}}$ with Dirichlet boundary condition on $\omega$ is denoted by $\lambda_1(\omega)$.
\item We often write
\[
f(0^+):= \lim_{r \to 0^+} f(r) \quad \text{and} \quad f(+\infty):= \lim_{r \to +\infty} f(r)
\]
if the limits exist.
\end{itemize}

In the paper we consider two classes of variational problems: regular ones of type
\begin{equation}\label{regular pb}
\begin{cases}
-\Delta u_i = -\beta \sum_{j \neq i} u_j^2 u_i & \text{in $\R^N$} \\
u_i \ge 0 & \text{in $\R^N$},
\end{cases}
\end{equation}
where $\beta>0$, and segregated ones of type
\begin{equation}\label{segregated pb}
\begin{cases}
-\Delta v_i = 0 & \text{in $\{v_i>0\}$} \\
v_i \ge 0 & \text{in $\R^N$}, \\
v_i v_j \equiv 0 & \text{in $\R^N$ for every $i \neq j$}.
\end{cases}
\end{equation}
We introduce suitable \emph{Almgren  frequency functions} according to whether we are considering \eqref{regular pb} or \eqref{segregated pb}. If $\mf{u}$ is a solution of \eqref{regular pb}, for $x_0 \in \R^N$ and $r>0$ we define 
\begin{equation}\label{def N regular}
\begin{split}
  \bullet \quad & H(\mf{u},x_0,r):= \frac{1}{r^{N-1}} \int_{\partial B_r(x_0)} \sum_{i=1}^k u_i^2 \\
  \bullet \quad & E(\mf{u},x_0,r):= \frac{1}{r^{N-2}} \int_{B_r(x_0)} \sum_{i=1}^k |\nabla u_i|^2+ \beta \sum_{1\le i<j\le k} u_i^2 u_j^2 \\ 
  \bullet \quad & N(\mf{u},x_0,r):= \frac{E(\mf{u},x_0,r)}{H(\mf{u},x_0,r)} \qquad (\text{Almgren frequency function}).
\end{split}  
    \end{equation}
If $\mf{v}$ is a solution of \eqref{segregated pb}, for $x_0 \in \R^N$ and $r>0$ we set 
\begin{equation}\label{def N segregated}
\begin{split}
  \bullet \quad & \tilde E(\mf{v},x_0,r):= \frac{1}{r^{N-2}} \int_{B_r(x_0)} \sum_{i=1}^k |\nabla v_i|^2 \\ 
  \bullet \quad & \tilde N(\mf{v},x_0,r):= \frac{\tilde E(\mf{v},x_0,r)}{ H(\mf{v},x_0,r)} \qquad (\text{Almgren frequency function}),
  \end{split}
\end{equation}
where $H(\mf{v},x_0,r)$ is defined as in the first one of the \eqref{def N regular}.

Let $\mf{u}$ be a solution of \eqref{system} (or to \eqref{segregated pb}). We set, for $x_0 \in \R^N$ and $R>0$,
\[
\mf{u}_{x_0,R}(x):=\frac{\mf{u}(x_0+Rx)}{H(\mf{u},x_0,R)^{1/2}} . 
\]
We are interested in the asymptotic behaviour of the \emph{blow-down family} $\{\mf{u}_{x_0,R}\}$ as $R \to +\infty$. We mainly consider the case $x_0=0$, writing simply $\mf{u}_{R}$ instead of $\mf{u}_{0,R}$ to simplify the notation.

The first of our main results is the extension of the blow-down Theorem \ref{thm blow down 2 comp} in the present setting. 
%To this aim, it is convenient to distinguish between the cases $N=2$ or $N \ge 3$. 

\begin{theorem}\label{them: blow down k dim 2}
Let $N,k\ge 2$, let $\mf{u}$ be a nonnegative solution of \eqref{system}, and let us assume that 
\[
\lim_{r \to +\infty} N(\mf{u},0,r) =: d < +\infty.
\] 
Then, up to a subsequence, 
\[
\mf{u}_R \to \mf{u}_\infty =r^d(g_1(\theta), \dots, g_k(\theta)) \quad \text{as $R \to +\infty$}
\]
in $\mathcal{C}^0_{\loc}(\R^N)$ and in $H^1_{\loc}(\R^N)$, where $(r,\theta) \in [0,+\infty) \times \S^{N-1}$ is a system of polar coordinates in $\R^N$ centred in $0$. Furthermore:
\begin{itemize}
\item the components $u_{i,\infty}$ are nonnegative and with disjoint support: $u_{i,\infty} u_{j,\infty} \equiv 0$ for every $i \neq j$;
\item  $\Delta u_{i,\infty}=0$ in the positivity domain $\{u_{i,\infty} >0\}$;
\item if for some $i \neq j$ there exists two adjacent nodal domains $B_i \subset \{u_{i,\infty}>0\}$ and $B_j \subset \{u_{j,\infty}>0\}$, then $u_{i,\infty}-u_{j,\infty}$ is harmonic in $\Int(\overline{B_i \cup B_j})$;
\item the set $\{\mf{u}_\infty=\mf{0}\} \cap \pa B_1$ has null $(N-1)$-dimensional measure;
\item $H(\mf{u},0,R) R^2 \sum_{i < j} u_{i,R}^2 u_{j,R}^2 \to 0$ in $L^1_{loc}(\R^N)$.
\end{itemize}
Let now $N=2$. Then, in addition, $d$ is a half-integer. Moreover, letting
\[
\Psi_d(r,\theta):= \frac{1}{\sqrt{\pi}} r^d \sin\left(d\theta\right),
\]
%be the homogeneous harmonic function of degree $d$, positive in the cone $\Theta_d:= \{r>0, 0<\theta<\pi/d\}$ and vanishing on the boundary $\pa \Theta_d$.
there exists a partition $(A_1,\dots,A_k)$ of the positivity domain $\Sigma_{|\Psi_d|}=\{|\Psi_d| > 0\}$, where for every $i$
\[
\text{$A_i$ is the union of non-adjacent nodal domains of $\Sigma_{|\Psi_d|}$},
\]
such that, up to a subsequence and up to a rotation, $\mf{u}_R \to (\chi_{ A_1},\dots,\chi_{A_k})|\Psi_d|$ as $R \to +\infty$, in $\mathcal{C}^0_{loc}(\R^N)$ and in $H^1_{loc}(\R^N)$. 
\end{theorem}

\begin{remark}
1) The same result holds for blow-down sequences centred at $x_0 \neq 0$. \\
2) By Proposition 5.2 in \cite{BeTeWaWe} (reported in Subsection \ref{sub: Almgren}), it follows that the limit $N(\mf{u},0,+\infty)$ always exists. Moreover, it is finite if and only if $\mf{u}$ has algebraic growth. 
\end{remark}

Theorem \ref{thm blow down 2 comp} with $N=2$ is a particular case of Theorem \ref{them: blow down k dim 2}; note that when $k$ is odd we have to take into account the possibility that the homogeneity degree of the limiting profile is a half-integer; this is coherent with Theorem \ref{thm: ex tante comp}. Let us also observe that, when $k$ is odd, $\Psi_d$ does not define a harmonic function in $\R^2$ in polar coordinates, since it is not $2\pi$-periodic in $\theta$; it can be seen as a harmonic function in the double covering $\{r \ge 0, 0 \le \theta < 4\pi\}$. 

The blow-down theorem will be the starting point in the derivation of the desired classification results. To this aim, we emphasize the relation between the growth rate of a solution and its Almgren frequency function.

\begin{proposition}\label{prop: sui growth rates}
Let $\mf{u}$ be a nonnegative solution of \eqref{system} having algebraic growth. Then $d:= N(\mf{u},0,+\infty) \in (0,+\infty)$ is the growth rate of $\mf{u}$.
\end{proposition}

This proposition implies that any solution of \eqref{system} having algebraic growth has a growth rate $d:= N(\mf{u},0,+\infty)$. The strategy we shall adopt to prove our Liouville-type theorems rests on the idea that $d$ characterizes the asymptotic profile of the solution by means of Theorem \ref{them: blow down k dim 2}: in particular, the value $d$ characterizes the maximal number of non-trivial components for a limiting profile, which hopefully should coincide with the maximal number of non-trivial components of the ``original" solution. In this perspective, the main difficulty is represented by the lack of uniqueness of the asymptotic profile (the convergence in Theorem \ref{them: blow down k dim 2} takes place only up to a subsequence), and in general by the difficulty in deriving rigorous information on the ``original" solution starting from the knowledge of the blow-down limit (we refer the interested reader to \cite{FaSo} and \cite{Wa}, where these problems are sources of tremendous complications). We can overcome these obstructions by means of the following intermediate result, which holds in any dimension.

\begin{proposition}\label{thm: zero total}
Let $N \ge 2$, and let $\mf{u}$ be a nonnegative solution of \eqref{system} having algebraic growth. Let us assume that there exists a sequence $R_n \to +\infty$ as $n \to \infty$, such that $u_{i,R_n} \to 0$ in $\mathcal{C}^0_{\loc}(\R^N)$ for some $i$. Then $u_i \equiv 0$ in $\R^N$. 
\end{proposition}

Thanks to Proposition \ref{thm: zero total}, we prove the Liouville-type Theorem \ref{protothm 1} in dimension $2$. In terms of the Almgren frequency function, it can be re-phrased as follows.

\begin{theorem}\label{thm: Liouville dim 2}
Let $N=2$, $k \ge 2$, and let $\mf{u}=(u_1,\dots,u_k)$ be a nonnegative solution of \eqref{system} such that $N(\mf{u},0,+\infty) =: d \in (0,+\infty)$. Then at most $2d$ components of $\mf{u}$ do not vanish identically. \\
Equivalently, let $N=2$, $k \ge 2$, and let $\mf{u}=(u_1,\dots,u_k)$ be a positive solution of \eqref{system} such that $N(\mf{u},0,+\infty) =: d \in (0,+\infty)$. Then $d \ge k/2$.
\end{theorem}

In light of Theorem \ref{thm: ex tante comp}, the result gives a sharp estimate in dimension $2$ on the minimal admissible growth rates for positive solutions of \eqref{system} with a given number of components. The higher dimensional case is more involved, reflecting the impossibility of deriving a complete description of the admissible limiting profile for solutions of \eqref{system}, see Theorem \ref{them: blow down k dim 2}.  In connection to this, we point out that all the existence results available in the literature have been achieved in dimension $2$ (thus leading to $2$-dimensional solutions of \eqref{system} in any dimension $N \ge 2$); so far it is still unknown if true $N$-dimensional solutions of \eqref{system} with $N \ge 3$ exist and can exhibit different asymptotic behaviour with respect to the $2$-dimensional case. Writing ``true $N$-dimensional solutions" we refer to solutions in dimension $N$ which cannot be obtained by solutions in dimension $N-1$ adding the dependence on $1$ variable (up to a rotation). Nevertheless even in higher dimension not all is lost: by means of Theorem \ref{them: blow down k dim 2} and Proposition \ref{thm: zero total} we can relate the maximal number of nontrivial components of a solution of \eqref{system} having a prescribed growth with the solution of an optimal partition problem for the unitary sphere $\S^{N-1}$. 

\begin{definition}
Let $1 \le k \in \N$. A \emph{$k$-partition} (or, simply, \emph{partition}) of $\S^{N-1}$ is a family $\omega=(\omega_1,\dots,\omega_k)$ of mutually disjoint open and connected subsets $\omega_i \subset \S^{N-1}$. We denote the class of the $k$-partition of $\S^{N-1}$ as $\mathcal{P}_k(\S^{N-1})$.
\end{definition}

%We point out that in the literature, for instance in \cite{HeHHTer1,HeHHTer2}, also slightly more general definition of $k$-partition has been considered. However, the previous definition is sufficient for our aims. 
We define 
\begin{align}
\mathcal{L}_k(\S^{N-1}) := \inf_{\omega \in \mathcal{P}_k(\S^{N-1})} \max_{i=1,\dots,k} \lambda_1(\omega_i) \label{def di L_k} \\
\gamma(t) := \sqrt{\left(\frac{N-2}{2}\right)^2+t}-\left(\frac{N-2}{2}\right). \label{def di gamma}
\end{align}
Note that $\gamma$ is monotone increasing and is such that $\gamma(t) \to +\infty$ as $t \to +\infty$. The following is a more convenient statement of Theorem \ref{protothm 1} holding in any dimension. 

\begin{theorem}\label{thm: higher dim}
Let $N,k \ge 2$, and let $\mf{u}=(u_1,\dots,u_k)$ be a nonnegative solution of \eqref{system} such that $N(\mf{u},0,+\infty) =: d \in (0,+\infty)$. If $m$ is the maximal positive integer such that $\gamma(\mathcal{L}_m(\S^{N-1})) \le d$, then at most $m$ components of $\mf{u}$ do not vanish identically. \\
Equivalently, let $N,k \ge 2$, and let $\mf{u}$ be a positive solution of \eqref{system} such that $N(\mf{u},0,+\infty) =: d \in (0,+\infty)$. Then $d \ge \gamma(\mathcal{L}_k(\S^{N-1}))$.
\end{theorem}

\begin{remark}\label{rem: geometric connection}
It is worth to observe that the connection between positive solutions with minimal growth of system \eqref{system} and the optimal partition problem $\mathcal{L}_k(\S^{N-1})$ goes beyond the relation $d \ge \gamma(\mathcal{L}_k(\S^{N-1}))$. Indeed, in the proof of Theorem \ref{thm: higher dim}, we will show that if there exists a solution $\mf{u}$ of \eqref{system} having the minimal admissible growth $d= \gamma(\mathcal{L}_k(\S^{N-1}))$, then $\mf{u}$ is asymptotic to an optimal partition for $\mathcal{L}_k(\S^{N-1})$, in the following sense: given any blow-down limit $\mf{u}_\infty$ of $\mf{u}$, let $\omega_i:= \supp g_i \subset \S^{N-1}$, where $g_i$ is defined in Theorem 1.4. Then $(\omega_1,\dots,\omega_k)$ is an optimal partition for $\mathcal{L}_k(\S^{N-1})$.
\end{remark}

Theorem \ref{thm: higher dim} is the base point for the proof of Corollary \ref{protothm 22}. Again, we think that the following re-formulation is more suited to describe our result.

\begin{corollary}\label{protothm 2}
Let $N,k \ge 2$, and let $\mf{u}=(u_1,\dots,u_k)$ be a nonnegative solution of \eqref{system} such that $N(\mf{u},0,+\infty)=:d \in (0,+\infty)$. Then either $d=1$ or $d \ge 3/2$. Furthermore: \begin{itemize}
\item[($i$)] if $d=1$, then $\mf{u}$ has exactly $2$ non-trivial components; 
\item[($ii$)] if $d=3/2$, then $\mf{u}$ has exactly $3$ non-trivial components.
\end{itemize}
\end{corollary}

The proof of point ($i$) is obtained as a particular case of a more general result (see Theorem \ref{thm: general systems} and Proposition \ref{prop: Almgren crescita puntuale}), while the second part and the jump in the admissible values of $N(\mf{u},0,+\infty)$ require a careful further analysis which can be carried on only for solutions of system \eqref{system}. We emphasize that, in light of the known existence results for system \eqref{system} with $k=2$ or $k=3$, Corollary \ref{protothm 2} is optimal in any dimension. Moreover, in proving point ($ii$) we can determine the optimal value $\mathcal{L}_3(\S^{N-1})$ for every $N$, partially extending the main result in \cite{HeHHTer2}.

\begin{theorem}\label{thm: on optimal partition}
In any dimension $N \ge 3$, it results that
\[
\mathcal{L}_3(\S^{N-1})= \frac{3}{2}\left(\frac{3}{2}+N-2\right),
\]
and an optimal partition is the extension in dimension $N$ of the so-called $\mf{Y}$-partition of $\S^{N-1}$. 
\end{theorem}  

\begin{remark}
For the definition of the $\mf{Y}$-partition, we refer to \cite{HeHHTer2}. We point out that we do not prove the uniqueness of the generalized $\mf{Y}$-partition as a solution of $\mathcal{L}_3(\S^{N-1})$.
\end{remark}

The relation between optimal partition problems and Liouville-type theorems has been already observed e.g. in \cite{AltCafFri,CoTeVePDE,CoTeVe, NoTaTeVe}. We refer in particular to Proposition 7.1 in \cite{CoTeVe}, where the authors related the minimal growth of a positive solution of Lotka-Volterra type systems with the quantity
\begin{equation}\label{def di beta}
\beta(k,N):= \inf_{(\omega_1,\dots,\omega_k) \in \mathcal{P}_k(\S^{N-1})} \frac{2}{k} \sum_{i=1}^k \gamma\left( \lambda_1(\omega_i)\right)
\end{equation}
We think that it is remarkable to observe that the very same approach leads to a more general result, involving \emph{subsolutions} to a wide class of systems. Let us consider
\begin{equation}\label{general syst}
\begin{cases}
-\Delta u_i \le -u_i g_i(x,\mf{u}) & \text{in $\R^N$} \\
u_i \ge 0 & \text{in $\R^N$},
\end{cases} \qquad i=1,\dots,k,
\end{equation}
under the following assumptions on the nonlinear terms $g_i \in \mathcal{C}(\R^N \times [0,+\infty)^k)$:
\begin{itemize}
\item[($H1$)] $g_i(x,\mf{t}) \ge \underline{g_i}(\mf{t}) \ge 0$ for every $(x,\mf{t}) \in \R^N \times [0,+\infty)^k$, where $\underline{g_i} \in \mathcal{C}([0,+\infty)^k)$; 
\item[($H2$)] if $\underline{g}_i(\mf{t})=0$, then either $t_j= 0$ for every $j \neq i$, or $t_i=0$; 
\item[($H3$)] $g_i(x,\mf{t})$ is monotone non-decreasing in $t_j$ for every $j$.
\end{itemize}
As typical example, the reader may think at the case $g_i(x,\mf{t})=  \sum_{j \neq i} t_j^2$ defining system \eqref{system}, but even to more general interaction terms (neither necessarily variational, nor symmetric) such as
\[
g_i(x,\mf{t})= \sum_{j \neq i} a_{ij}(x) t_j^{p_j} t_i^{q_j},
\]
with $a_{ij}(x) \ge \underline{a}_{ij}>0$ in $\R^N$ and $p_j>0$, $q_j \ge 0$.

\begin{theorem}\label{thm: general systems}
Let $N,k \ge 2$. Under assumptions ($H1$)-($H3$), let $\mf{u}=(u_1,\dots,u_k) \in \mathcal{C}(\R^N) \cap H^1_{\loc}(\R^N)$ satisfy \eqref{general syst} and
\begin{equation}\label{eq 28 marzo}
u_1(x) +\dots +u_k(x) \le C(1+|x|^d) \qquad \text{for every $x \in \R^N$}
\end{equation}
for some $C>0$ and $d \ge 1$. Let $m$ be the maximal positive integer such that $\beta(m,N) \le 2d$. Then at most $m$ components of $\mf{u}$ do not vanish identically. \\
In other words, if $\mf{u}=(u_1,\dots,u_k) \in \mathcal{C}(\R^N) \cap H^1_{\loc}(\R^N)$ is a positive solution of \eqref{general syst} satisfying \eqref{eq 28 marzo}, then necessarily $\beta(k,N) \le 2d$.
\end{theorem}

System \eqref{system} fits in the assumptions of Theorem \ref{thm: general systems}. It is then straightforward to obtain the first part of Corollary \ref{protothm 2} as a particular case of a more general result.

\begin{corollary}\label{corol general}
Let $N,k\ge 2$. Let us assume that ($H1$)-($H3$) are satisfied, and let $\mf{u}=(u_1,\dots,u_k) \in \mathcal{C}(\R^N) \cap H^1_{\loc}(\R^N)$ satisfy \eqref{general syst}.
\begin{itemize}
\item[($i$)] If there exists $C>0$ such that
\[
u_1(x) +\dots +u_k(x) \le C(1+|x|) \qquad \text{for every $x \in \R^N$},
\]
then at most $2$ components of $\mf{u}$ do not vanish identically.
\item[($ii$)] If there exist $C>0$ and $\alpha \in (0,1)$ such that 
\[
u_1(x) +\dots +u_k(x) \le C(1+|x|^{\alpha}) \qquad \text{for every $x \in \R^N$},
\]
then at most $1$ component of $\mf{u}$ does not vanish identically.
\end{itemize}
\end{corollary}

For the proof it is sufficient to recall that $\beta(k,N)$ is monotone non-decreasing in $k$, and such that $\beta(k,N) >\beta(2,N)$ whenever $k \ge 3$ (see the inequality (31) in \cite{CoTeVe}). Moreover, $\beta(2,N)=2$ in any dimension $N$ (see \cite{Spe}).

\begin{remark}\label{rem: su k=2}
In \cite{CoTeVePDE} it has been proved that both $\beta(k,N)$ and $\mathcal{L}_{k}(\S^{N-1})$ are achieved. Here we used Theorem \ref{thm: general systems} instead of Theorem \ref{thm: higher dim} to prove point ($i$) in Corollary \ref{protothm 2}, but we point out that Theorem \ref{thm: higher dim} is stronger in the particular case of system \eqref{system}. Indeed it is well know that $\gamma(\mathcal{L}_k(\S^{N-1})) \ge \beta(k,N)/2$ for every $k,N$. Moreover, since the optimal value $\beta(2,N)$ is achieved by the equator-cut sphere (see \cite{Spe}), $\gamma(\mathcal{L}_2(\S^{N-1})) = \beta(2,N)/2 = 1$ for every $N$, which implies directly point ($i$) in Corollary \ref{protothm 2}.
\end{remark}

The last part of the paper is devoted to the $1$-dimensional symmetry of solutions of \eqref{system}. The proof of Theorem \ref{thm: 1-dimensional symmetry} consists in showing that, under the assumptions of both points ($i$) and ($ii$), only two components of the solution can be non-trivial, and thus the solution is $1$-dimensional thanks to the results in \cite{FaSo,Wa,Wa2}. If the solution has linear growth, the fact that $\mf{u}$ has at most two non-trivial components follows directly by Corollary \ref{protothm 2}; if $\mf{u}$ satisfies the assumption of point ($ii$), we at first study the asymptotic profile of the blow-down sequences $\{\mf{u}_{R_n}\}$, proving that any blow-down limit has two non-trivial components; then, to recover the result for $\mf{u}$, we apply the crucial Proposition \ref{thm: zero total}.

\subsection*{Structure of the paper}

In Section \ref{sec: preliminaries} we collect some results which will often be employed in the rest of the paper. Section \ref{sec: blow.down} is devoted to the proofs of Theorem \ref{them: blow down k dim 2} and of Proposition \ref{thm: zero total}. The proofs of the Liouville-type Theorems \ref{thm: Liouville dim 2} and \ref{thm: higher dim}, which concern system \eqref{system}, together with that of Corollary \ref{protothm 2} and of Theorem \ref{thm: on optimal partition}, are the object of Section \ref{sec: Liouv 2}. In Section \ref{sec: general} we consider the general system \eqref{general syst}, proving Theorem \ref{thm: general systems}. Finally, in Section \ref{sec: 1-dim} we address the problem of the $1$-dimensional symmetry, proving Theorem \ref{thm: 1-dimensional symmetry}.

\section{Preliminaries}\label{sec: preliminaries}

In what follows we recall some essentially known results which will be useful in the rest of the paper.

\subsection{Almgren monotonicity formulae}\label{sub: Almgren}
Here we recall some properties of the Almgren frequency function associated to solution of \eqref{system}, proving in particular Proposition \ref{prop: sui growth rates}.
\begin{proposition}[Proposition 5.2 in \cite{BeTeWaWe}]\label{prop: monotonicity formula}
Let $N \ge 2$, $x_0 \in \R^N$, and let $\mf{u}$ be a nonnegative solution to \eqref{system}. The Almgren frequency function $N(\mf{u},x_0,r)$ is monotone non-decreasing in $r$.
\end{proposition}

We infer the following doubling properties.

\begin{proposition}[Proposition 5.3 in \cite{BeTeWaWe}]\label{prop: doubling}
Let $\mf{u}$ be a nonnegative solution of \eqref{system}.
\begin{itemize}
\item[($i$)] For every $0<r_0 \le r_1 <r_2$ it results that
\[
\frac{H(\mf{u},x_0,r_2)}{r_2^{2N(\mf{u},x_0,r_0)}} \ge \frac{H(\mf{u},x_0,r_1)}{r_1^{2N(\mf{u},x_0,r_0)}}.
\]
\item[($ii$)] Assume that $N(\mf{u},x_0,r) \le d$ for every $r>0$. Then
\[
\frac{H(\mf{u},x_0,r_2)}{r_2^{2d}} \le e^d \frac{H(\mf{u},x_0,r_1)}{r_1^{2d}} 
\]
for every $0<r_1<r_2$.
\end{itemize}
\end{proposition}

The doubling properties allow us to relate the Almgren frequency function with the growth rate of the associated solution, as stated in Proposition \ref{prop: sui growth rates}.

\begin{proof}[Proof of Proposition \ref{prop: sui growth rates}]
If $d'>d$, then by Proposition \ref{prop: doubling}-($ii$) we have
\[
\frac{H(\mf{u},0,r)}{r^{2d'}} \le C \frac{r^{2d}}{r^{2d'}} \to 0 \qquad \text{as $r \to +\infty$}.
\]
If $d'<d$, by monotonicity there exists $\bar r>0$ such that $N(\mf{u},0,\bar r)=d'+\eps<d$ for some $\eps>0$. Then, by Proposition \ref{prop: doubling}-($i$), for every $r>\bar r$
\[
\frac{H(\mf{u},0,r)}{r^{2d'}} \ge C \frac{r^{2(d'+\eps)}}{r^{2d'}} \to +\infty \qquad \text{as $r \to +\infty$}. \qedhere
\]
\end{proof}

It is also possible to relate the Almgren quotient with a pointwise upper bound.

\begin{proposition}\label{prop: Almgren crescita puntuale}
$N(\mf{u},0,r) \le d$ for every $r>0$ if and only if there exists $C,d>0$ such that $\textstyle \sum_i u_i(x) \le C(1+|x|^d)$ in $\R^N$.
\end{proposition}
For the proof, see Lemma 2.1 in \cite{Fa} and Corollary A.8 in \cite{FaSo}.

Since the growth rate of a solution $\mf{u}$ of \eqref{system} coincides with the limit of the Almgren frequency function, it is natural to have the following result.

\begin{lemma}\label{lem: constancy of N}
Let $\mf{u}$ be a nonnegative solution of \eqref{system} having algebraic growth. \\
Then $N(\mf{u},x_0,+\infty)$ is constant as function of $x_0 \in \R^N$.
\end{lemma}
\begin{proof}
Let $N(\mf{u},0,+\infty)=d<+\infty$, and let us assume by contradiction that there exists $x_0 \neq 0$ such that $N(\mf{u},x_0,+\infty)=d' \neq d$. Firstly, $d'<+\infty$ since $\mf{u}$ has algebraic growth. Only to fix our minds, we suppose that $d'<d$, so that there exists $\eps>0$ such that $d'+\eps<d$. By Propositions \ref{prop: monotonicity formula} and \ref{prop: doubling} there exists $r_0>0$ such that $H(\mf{u},x_0,r) \le C r^{2d'}$ and $H(\mf{u},0,r) \ge C r^{2(d'+\eps)}$ for $r>r_0$. Therefore on one side
\[
\int_{B_{r_1}(x_0) \setminus B_{r_0}(x_0)} \sum_i u_i^2 = \int_{r_0}^{r_1} s^{N-1} H(\mf{u},x_0,s)\, \de s \le C r_1^{2d'+N}
\]
for $r_1>r_0$, while on the other side
\[
\int_{B_{r_2}(0) \setminus B_{r_0}(0)} \sum_i u_i^2 = \int_{r_0}^{r_2} s^{N-1} H(\mf{u},0,s)\, \de s \ge C r_2^{2(d'+\eps)+N}
\]
for $r_2>r_0$. As a consequence, for $r > r_0$, we have
\begin{multline*}
C r^{2(d'+\eps)+N} \le \int_{B_{r}(0) \setminus B_{r_0}(0)} \sum_i u_i^2 \\
\le \int_{B_{r+|x_0|}(x_0) \setminus B_{r_0}(x_0)} \sum_i u_i^2 + \int_{B_{r_0}(x_0) \setminus B_{r_0}(0)} \sum_i u_i^2 \\
\le C (r+|x_0|)^{2d'+N}+C \le C r^{2d'+N},
\end{multline*}
which gives a contradiction for $r$ sufficiently large.
\end{proof}

\subsection{Segregated configurations}\label{sub: richiami TaTe}

In Definition 1.2 of \cite{TaTe}, H. Tavares and the second author introduced the class of functions $\mathcal{G}(\Omega)$. We consider a subclass of particular interest in the present setting.

\begin{definition}\label{def: G *}
For an open set $\Omega \subset \R^N$, we define the class $\mathcal{G}(\Omega)$ of nontrivial functions $\mf{0} \neq \mf{v}=(v_1,\dots,v_k)$ whose components are nonnegative and locally Lipschitz continuous in $\Omega$, and such that the following properties holds:
\begin{itemize}
\item $v_i v_j \equiv 0$ in $\Omega$ for every $i \neq j$;
\item for every $i$ 
\[
-\Delta v_i= -\mu_i \qquad \text{in $\Omega$ in distributional sense},
\]
where $\mu_i$ is a nonnegative Radon measure supported on the set $\pa\{v_i>0\}$;
\item defining for $x_0 \in \Omega$ and $r>0$ such that $B_r(x_0) \subset \Omega$ the function $\tilde E(\mf{v},x_0,r)$ as in \eqref{def N segregated}, 
we assume that $\tilde E$ is absolutely continuous as function of $r$ and 
\[
\frac{d}{dr} \tilde E(\mf{v},x_0,r)= \frac{1}{r^{N-2}} \int_{\pa B_r(x_0)} \sum_{i=1}^k (\pa_\nu v_i)^2;
\]
\end{itemize}
We write that $\mf{v} \in \mathcal{G}_{\loc}(\R^N)$ if $\mf{v} \in \mathcal{G}(B_R)$ for every $R>0$. We write that $\mf{v} \in \mathcal{G}_{\loc}^*(\R^N)$ if  $\mf{v} \in \mathcal{G}_{\loc}(\R^N)$ and is homogeneous with respect to some $x_0 \in \R^N$, in the sense that there exists $\gamma>0$ such that $\mf{v}(r,\theta)=r^\gamma \mf{g}(\theta)$, where $(r,\theta)$ is a system of polar coordinates in $\R^N$ centred in $x_0$.
\end{definition}

It is possible to introduce an Almgren frequency function associated to any $\mf{v} \in \mathcal{G}_{\loc}^*(\R^N)$ as in \eqref{def N segregated}, and to prove a monotonicity formula for it (see Theorem 2.2 and Remark 2.4 in \cite{TaTe}).

\begin{proposition}\label{prop: monot segregated}
Let $\mf{v} \in \mathcal{G}_{\loc}(\R^N)$. For $x_0 \in \R^N$ and $r>0$, the function $\tilde N(\mf{v},x_0,r)$ is non-decreasing in $r$. Moreover, $\tilde N(\mf{v},x_0,r)=const.= \sigma>0$ if and only if $\mf{v}(r,\theta)= r^\sigma\mf{g}(\theta)$, where $(r,\theta)$ denotes a system of polar coordinates centred in $x_0$.
\end{proposition}

As a consequence, it is possible to derive doubling properties similar to those of Proposition \ref{prop: doubling}, and to recast the proof of Lemma \ref{lem: constancy of N} in the present setting, obtaining the following statement.

\begin{lemma}\label{lem: constancy of N segregated}
Let $\mf{v} \in \mathcal{G}_{\loc}(\R^N)$ having algebraic growth. Then $N(\mf{v},x_0,+\infty)$ is constant as function of $x_0 \in \R^N$.
\end{lemma}

We conclude this subsection with a definition.

\begin{definition}\label{def: mult}
Let $\mf{v} \in \mathcal{G}_{\loc}(\R^N)$, and let $x_0 \in \{\mf{v}=\mf{0}\}$. We define the multiplicity of $x_0$ as
\[
\# \left\{i=1,\dots,k: \text{ for every $r>0$ it results $B_r(x_0) \cap \{v_i>0\} \neq \emptyset$}\right\}.
\]
\end{definition}

\subsection{Decay estimates}

If $(u,v)$ solves \eqref{system} and $u$ is large in a ball $B_{2r}(x_0)$, then by comparison principles $v$ has to be exponentially small with respect to $u$ in a smaller ball.

\begin{lemma}[Lemma 4.4 in \cite{CoTeVe}]\label{lem: exp decay}
Let $x_0 \in \R^N$ and $r>0$. Let $v \in \mathcal{C}(B_{2r}(x_0) \cap H^1(B_{2r}(x_0))$ be such that
\[
\begin{cases}
-\Delta v \le -K v & \text{in $B_{2r}(x_0)$} \\
v \ge 0 & \text{in $B_{2r}(x_0)$} \\
v \le A & \text{on $\pa B_{2r}(x_0)$},
\end{cases}
\]
where $K$ and $A$ are two positive constants. Then there exists $C>0$ depending only on the dimension $N$ such that
\[
\sup_{x \in B_r(x_0)} v(x) \le C A  e^{-C K^{1/2} r}.
\]
\end{lemma}

\section{Asymptotic behaviour of positive solution}\label{sec: blow.down}

In this section we prove Theorem \ref{them: blow down k dim 2} and Proposition \ref{thm: zero total}.

\subsection{Blow-down limits}

\begin{proof}[Proof of Theorems \ref{them: blow down k dim 2}]
%We introduce quantities similar to $H$, $E$ and $N$ for elements of the blow-down family:
%\begin{align*}
%  \bullet \quad & H_R(r):= \frac{1}{r^{N-1}} \int_{\partial B_r} \sum_{i=1}^k u_{i,R}^2 \\
%  \bullet \quad & E_R(r):= \frac{1}{r^{N-2}} \int_{B_r} \sum_{i=1}^k |\nabla u_{i,R}|^2+ H(R) R^2 \sum_{1\le i<j\le k} u_{i,R}^2 u_{j,R}^2 \\ 
%  \bullet \quad & N_R(r):= \frac{E_R(r)}{H_R(r)}.
%\end{align*}
The elements of the blow-down family satisfy
\[
 \begin{cases}
  -\Delta u_{i,R}=-\sum_{j \neq i} R^2 H(\mf{u},0,R) u_{j,R}^2 u_{i,R} & \text{in $\R^N$} \\
u_{i,R} \ge 0 & \text{in $\R^N$} 
 \end{cases} \qquad i=1,\dots,k,
\]
so that are well defined functions $H(\mf{u}_R,0,r)$, $E(\mf{u}_R,0,r)$ and $N(\mf{u}_R,0,r)$ as in \eqref{def N regular}. By direct computations it is easy to check that 
\[
H(\mf{u}_R,0,r)=\frac{H(\mf{u},0,rR)}{H(\mf{u},0,R)}, \quad  E(\mf{u}_R,0,r)=\frac{E(\mf{u},0,rR)}{H(\mf{u},0,R)}, \quad  N(\mf{u}_R,0,r)=N(\mf{u},0,rR).
\] 
By the doubling property ($i$) in Proposition \ref{prop: doubling} $H(\mf{u},0,R) R^2 \to +\infty$ as $R \to +\infty$. Furthermore, by definition $H(\mf{u}_R,0,1)=1$ for every $R>0$, and by the Almgren monotonicity formula $N_R(r) \le d$ for every $r,R>0$. As a consequence, the doubling property ($ii$) in Proposition \ref{prop: doubling} implies that
\[
H(\mf{u}_R,0,r)\le  e^d r^{2d} 
\]
for every $R,r>1$. Hence, by subharmonicity, $\{\mf{u}_R\}$ is uniformly bounded in $L^\infty_{\loc}(\R^N)$, and we are in position to apply the local version of the main results in \cite{NoTaTeVe} (for the local version, we refer to Theorem 2.6 in \cite{Wa}): up to a subsequence $\mf{u}_R \to \mf{u}_\infty$ in $\mathcal{C}^0_{\loc}(\R^N)$ and in $H^1_{\loc}(\R^N)$, where by Corollary 8.3 in \cite{TaTe} $\mf{u}_\infty \in \mathcal{G}_{\loc}(\R^N)$. In particular 
\begin{itemize}
\item $u_{i,\infty} u_{j,\infty} \equiv 0$ in $\R^N$ for every $i \neq j$, and 
\[
\lim_{R \to +\infty} \int_{B_r} H(\mf{u},0,R) R^2 \sum_{i<j} u_{i,R}^2 u_{j,R}^2 =0 \qquad \text{for every $r>0$};
\]
\item $u_{i,\infty}$ is subharmonic in $\R^N$, and $\Delta u_{i,\infty} = 0$ in $\{u_{i,\infty}>0\}$, for every $i=1,\dots,k$.
\end{itemize}
%We introduce the Almgren frequency function $\tilde N(\mf{u}_\infty,0,r)$ as in \eqref{def N segregated}. Since $\mf{u}_\infty \in \mathcal{G}_{\loc}(\R^N)$ it is monotone non-decreasing in $r$, and the proof of the monotonicity formula reveals that it can be constant only if
%\begin{equation}\label{cause homogeneity}
%\left(\int_{\pa B_r} \sum_i \pa_\nu u_{i,\infty}^2 \right)\left( \int_{\pa B_r} \sum_i u_{i,\infty}^2 \right)- \left( \int_{\pa B_r} \sum_i u_{i,\infty} \pa_\nu u_{i,\infty} \right) = 0
%\end{equation}
%for every $r>0$. On the other hand, by the convergence 
Let $\tilde N(\mf{u}_\infty,0,r)$ be define in \eqref{def N segregated}. For every $r>0$
\[
\tilde N(\mf{u}_\infty,0,r) = \lim_{R \to +\infty} N(\mf{u}_R,0,r) = \lim_{R \to +\infty} N(\mf{u},0,Rr) = d
\]
so that as stated in Proposition \ref{prop: monot segregated} $\mf{u}_\infty(r,\theta)=r^d \mf{g}(\theta)$.
%\eqref{cause homogeneity} is in force, implying that there exists a radial function $C(r)$ such that
%\[
%\pa_r \mf{u}_{\infty}(r,\theta) = C(r) \mf{u}_{\infty}(r,\theta)
%\]
%for every $r>0$ and $\theta \in \S^{N-1}$. In turn, this gives $\mf{u}_{\infty}(r,\theta)=f(r) \mf{g}(\theta)$, and by harmonicity and constancy of the Almgren quotient $\tilde N(\mf{u}_\infty,0,\cdot)$ we deduce that $f(r) = r^d$.
Concerning the functions $g_i$, they have disjoint support and are such that $r^d g_i(\theta)$ is harmonic in $\{u_{i,\infty}>0\}$. Let us note that if $\{u_{i,\infty}>0\} \neq \emptyset$, then it is a cone, and by Theorem 1.1 in \cite{TaTe} the zero level set $\{\mf{u}_\infty=\mf{0}\}$ has null $N$-dimensional measure. By homogeneity, this implies that $\{\mf{u}_\infty=\mf{0}\} \cap \pa B_1$ has null $(N-1)$-dimensional measure. In what follows, we use the notation
\[
\{u_{i,\infty}>0\}= A_i = B_{1,i} \cup \cdots \cup B_{h_i,i},
\]
denoting by $B_{l,i}$ the nodal domains of $u_{i,\infty}$. If $B_{i,l_i}$ and $B_{j,l_j}$ are adjacent, then the reflection law proved in Theorem 1.1 in \cite{TaTe} implies that $u_{i,\infty} -u_{j,\infty}$ is harmonic in $\Int(\overline{B_{i,l_i} \cup B_{j,l_j}})$; the main result in Section 10 of \cite{DaWaZh} rules out the existence of point of multiplicity $1$; in other words, if there exist two non-empty connected components of some $A_i$ ($i=1,\dots,k$), then they have to be non-adjacent. \\
What we proved so far holds in any dimension $N \ge 2$. In what follows, we focus on the case $N=2$. By Theorem 1.1 in \cite{TaTe} and by homogeneity, the nodal set $\{\mf{u}_\infty=\mf{0}\}$ is the union of straight lines passing through the origin and meeting with equal angles. For some $i \neq j$, let $B_{i,l_i}$ and $B_{j,l_j}$ be adjacent nodal domains of $u_{i,\infty}$ and $u_{j,\infty}$, respectively. Then, as already noticed, $u_{i,\infty}-u_{j,\infty}$ is harmonic in the cone $\Int(\overline{B_{i,l_i} \cup B_{j,l_j}})=\{r >0, \theta_0<\theta<\theta_1\}$ (where $0 \le \theta_0 <\theta_1 <2\pi$); up to a rotation, it is not restrictive to assume that $\theta_0=0$, so that $w:=g_i-g_j$ satisfies
\[
\begin{cases}
w''+ d^2 w=0 & \text{in $(0,\theta_1)$} \\
w(0)=w(\theta_1/2)=w(\theta_1)=0 \\
w>0 \quad \text{in $(0,\theta_1/2)$}, \quad w<0 &\text{in $(\theta_1/2,\theta_1)$}
\end{cases}
\]
for some $\theta_1 \in (0,2\pi)$. It is straightforward to deduce that for some $C>0$ we have $w(\theta)=C \sin(d \theta)$ and $\theta_1= 2\pi/d$. Iterating this line of reasoning for any pair of adjacent nodal domains $B_{i,l_i}$ and $B_{j,l_j}$, and recalling that the functions $g_i$ are segregated and nonnegative, we conclude that $d \in \N/2$, and there exists a unique $C>0$ such that
\[
(g_1(\theta),\dots,g_k(\theta)) = (\chi_{A_1},\dots,\chi_{A_k}) C \sin(d\theta),
\]
which is uniquely determined as $C=1/\sqrt{\pi}$ by the normalization
\[
\int_0^{2\pi} C^2\sin^2(d\theta)\, \de\theta= 1.
\]
This completes the proof.
\end{proof}

\begin{remark}\label{rem: vale TaTe sui blow-down}
As observed, any blow-down limit of an arbitrary solution $\mf{u}$ of system \eqref{system} belongs to $\mathcal{G}^*_{\loc}(\R^N)$, see Definition \ref{def: G *}. Therefore, the results proved in \cite{TaTe} hold for the blow-down limits. This will be used in Subsection \ref{sub: 3 comp}. 
\end{remark}

\subsection{Proof of Proposition \ref{thm: zero total}} 

We aim at showing that if in the blow-down family one component $u_i$ vanishes along one sequence $R_n \to +\infty$, then it is identically zero. We reach this result through a series of lemmas.

We introduce the family $\mathcal{BD}_{\mf{u}}$ of the blow-down limits for a fixed nonnegative solution $\mf{u}$ of \eqref{system}: $\mf{w} \in \mathcal{BD}_{\mf{u}}$ if there exists a sequence $R_n \to +\infty$ such that $\mf{u}_{R_n} \to \mf{w}$ as $n \to \infty$ in $\mathcal{C}_{\loc}^0(\R^N)$ and in $H^1_{\loc}(\R^N)$. 

In what follows we derive some useful properties of $\mathcal{BD}_{\mf{u}}$. Here and in the rest of the section $d$ denotes the limit of the Almgren frequency function of the considered solution: $d:= N(\mf{u},0,+\infty)$, which is finite since $\mf{u}$ has algebraic growth. In several cases we consider blow-down sequences $\mf{u}_{R_n}$ converging to some limiting profile $\mf{u}_\infty \in \mathcal{BD}_{\mf{u}}$. Clearly the convergence has to be understood in $\mathcal{C}^0_{\loc}(\R^N)$ and in $H^1_{\loc}(\R^N)$, and we often omit this piece of information.

\begin{lemma}\label{prop BD}
($i$) The set $\mathcal{BD}_{\mf{u}}$ is locally uniformly bounded in the $1/2$-H\"older norm.\\
%For every $r>0$, there exists $C>0$ depending only on $r$ such that
%\[
%\sup_{i=1,\dots,k} \sup_{x,y \in B_r} |w_i(x)-w_i(y)| \le C|x-y|^{1/2} \qquad \text{for every $\mf{w} \in \mathcal{BD}_{\mf{u}}$}.
%\]
($ii$) The set $\mathcal{BD}_{\mf{u}}$ is closed under locally uniform convergence. \\
($iii$) For any $\mf{w} \in \mathcal{BD}_{\mf{u}}$ it results that 
\[
\int_{\pa B_r} \sum_{i=1}^k w_i^2 \le e^d r^{2d+N-1} \qquad \text{for every $r>1$}.
\]
\end{lemma}
\begin{proof}
($i$) It is a consequence of the local version of the main results proved in \cite{NoTaTeVe}, see Theorem 2.6 in \cite{Wa}. \\
($ii$) Let $\{\mf{w}_n\} \subset \mathcal{BD}_{\mf{u}}$ such that $\mf{w}_n \to \mf{w}$ in $\mathcal{C}^0_{\loc}(\R^N)$ as $n \to \infty$: given $r$ and $\eps>0$, there exists $\bar n \in \N$ such that
\begin{equation}\label{18-02eq1}
n>\bar n \quad \Longrightarrow \quad \sup_{i=1,\dots,k} \sup_{B_r} |w_{i,n}-w_i|<\frac{\eps}{2}.
\end{equation}
Since $\mf{w}_n \in \mathcal{BD}_{\mf{u}}$, there exists a sequence $R^n_m \to +\infty$ as $m \to \infty$ such that $\mf{u}_{R^n_m} \to \mf{w}_n$ in $\mathcal{C}^0_{\loc}(\R^N)$ and in $H^1_{\loc}(\R^N)$ as $m \to \infty$. Therefore, there exists $\bar m(n) \in \N$ such that 
\begin{equation}\label{18-02eq2}
m>\bar m(n) \quad \Longrightarrow \quad \sup_{i=1,\dots,k} \sup_{B_r} |u_{i,R_m^n}-w_{i,n}|<\frac{\eps}{2}.
\end{equation}
Now, for every $n$ let us choose $m_n>\bar m(n)$ so large that $R_n:= R^n_{m_n}$ tends to $+\infty$ as $n \to \infty$. Considering the blow-down sequence $\{\mf{u}_{R_n}\}$, it is easy to check that it is locally uniformly convergent to $\mf{w}$: indeed given $r,\eps>0$, by \eqref{18-02eq1} and \eqref{18-02eq2} we obtain
\[
\sup_{B_r} |u_{i,R_n}-w_{i}| \le \sup_{B_r} |u_{i,R_n}-w_{i,n}| + \sup_{B_r} |w_{i,n}-w_i|<\eps
\]
for every $i=1,\dots,k$ and $n > \bar n$.\\
($iii$) It follows from the doubling property ($ii$) in Proposition \ref{prop: doubling}, which holds for any element of the blow-down family and is stable under uniform convergence.
\end{proof}

In the next lemma we show that, under the assumptions of Proposition \ref{thm: zero total}, for the entire blow-down family the component $u_{i,R}\to 0$ as $R \to +\infty$.

\begin{lemma}\label{lem: zero in all the blow down}
Let us assume that there exists a sequence $R_n \to +\infty$ as $n \to \infty$, such that $u_{i,R_n} \to 0$ in $\mathcal{C}^0_{\loc}(\R^N)$ and in $H^1_{\loc}(\R^N)$. Then for every sequence $R_n' \to +\infty$ it results that $u_{i,R_n'} \to 0$ in $\mathcal{C}^0_{\loc}(\R^N)$ and in $H^1_{\loc}(\R^N)$ as $n \to \infty$.
\end{lemma}
\begin{proof}
We separate the proof in two steps. 

\emph{Step 1) There exists $\bar C>0$ such that for any $\mf{w} \in \mathcal{BD}_{\mf{u}}$, for any $i=1,\dots,k$ and for any (not empty) connected component $\omega_i$ of $\{w_i>0\} \cap \pa B_1$ it results 
\[
\int_{\omega_i}  w_i^2 \ge \bar C.
\]}
Assume by contradiction that the claim is not true. Then there exist a sequence $\{\mf{w}_n\} \subset \mathcal{BD}_{\mf{u}}$, and a sequence of connected components $\omega_{i_n,n}$ of $\{w_{i_n}>0\} \cap \pa B_1$, such that
\begin{equation}\label{abs1}
\int_{\pa B_1} \chi_{\omega_{i_n,n}} \sum_{j=1}^k w_{j,n}^2 = \int_{ \omega_{i_n,n}} w_{i_n}^2 \to 0 \qquad \text{as $ \to \infty$}.
\end{equation}
Up to a subsequence, we can assume that $i_n=i$. By the properties of the blow-down limits $\mf{w}_n(r,\theta)=r^d \mf{g}(\theta)$, where $g_{1,n},\dots,g_{k,n}$ have disjoint supports. Moreover, since $w_{i,n}$ has to be harmonic in its positivity domain 
\[
\begin{cases}
-\Delta_{\S^{N-1}} g_{i,n}= d(d+N-2) g_{i,n} & \text{in $\omega_{i,n}$} \\
g_{i,n}>0 & \text{in $\omega_{i,n}$} \\
g_{i,n} =0 & \text{on $\pa \omega_{i,n}$}.
\end{cases}
\]
This reveals that $d(d+N-2)$ is the first eigenvalue of $-\Delta_{\S^{N-1}}$ with Dirichlet boundary conditions on $\omega_{i,n}$, with corresponding eigenfunction $g_{i,n}$, and in particular there exists $C>0$ such that
\begin{equation}\label{measure not zero}
\text{$\mathcal{H}^{N-1}(\omega_{i,n}) \ge C$} \qquad \text{for every $n$}.
\end{equation}
This fact follows from the known dependence on $\lambda_1(\omega)$ on the measure of $\omega$, and in particular by the property
\[
\lambda_1(\omega) \to +\infty \quad \text{as}  \quad \mathcal{H}^{N-1}(\omega) \to 0. 
\]
Now, point ($iii$) of Lemma \ref{prop BD} implies that
\[
\int_{\pa B_r} \sum_{j=1}^k w_{i,n}^2 \le  e^d r^{2d+N-1} 
\]
for every $r>1$, for every $n$. By subharmonicity, $\{\mf{w}_n\}$ is uniformly bounded in $L^\infty_{\loc}(\R^N)$, and by point ($i$) of Lemma \ref{prop BD} it is also equi-continuous. Thus, by the Ascoli-Arzel\`a theorem it is locally uniformly convergent u.t.s. to some $\mf{w}$, still belonging to $\mathcal{BD}_{\mf{u}}$ thanks to point ($ii$) of Lemma \ref{prop BD}. As a consequence, by Theorem \ref{them: blow down k dim 2} the set $\{\mf{w}=\mf{0}\} \cap \pa B_1$ has null $(N-1)$-dimensional measure, and 
\begin{equation}\label{har not zero}
\sum_{j=1}^k w_{j,n}^2 \not \to 0\qquad \text{a.e. in $\S^{N-1}$}.
\end{equation} 
On the other hand, by \eqref{abs1}, up to a subsequence
\begin{equation}\label{har zero}
\chi_{\omega_{i,n}}  \sum_{j=1}^k w_{j,n}^2 \to 0\qquad \text{a.e. in $\S^{N-1}$}.
\end{equation} 
A comparison between \eqref{har not zero} and \eqref{har zero} implies that $\chi_{\omega_{i,n}} \to 0$ a.e. in $\mathbb{S}^{N-1}$, which by the dominated convergence theorem provides $\mathcal{H}^{N-1}(\omega_n) \to 0$ as $n \to \infty$, in contradiction with \eqref{measure not zero}.

\emph{Step 2) Conclusion of the proof.} \\
Assume by contradiction that there exist $R_n' \to +\infty$ such that $u_{i,R_n'} \not \to 0$ as $n \to \infty$. Arguing as in the proof of Theorem \ref{them: blow down k dim 2}, u.t.s. $\mf{u}_{R_n'} \to \mf{w}' \in \mathcal{BD}_{\mf{u}}$ in $\mathcal{C}^0_{\loc}(\R^N)$. By assumption $w_i' \not \equiv 0$, so its support $A_i$ is not empty, and by the first step
\[
\lim_{n \to \infty} \int_{\pa B_1} u_{i,R_n'}^2 \ge \frac{1}{2}\int_{\tilde A_i} w_i^2 \ge \frac{\bar C}{2}, 
\]
where $\tilde A_i=A_i \cap \pa B_1$. Extracting if necessary further subsequences, it is possible to sort the terms of $(R_n)$ (the sequence given by the assumption) and $(R_n')$ in such a way that $R_n \le R_n' \le R_{n+1}$ for every $n$. Note that, at least for $n$ sufficiently large,
\[
 \int_{\pa B_1} u_{i,R_n}^2 < \frac{\bar C}{8} \quad \text{and} \quad
 \int_{\pa B_1} u_{i,R_n'}^2 > \frac{3\bar C}{8},
\] 
so that thanks to the mean value theorem for every $n$ there exists $R_n'' \in (R_n,R_n')$ such that 
\[
  \int_{\pa B_1} u_{i,R_n''}^2 = \frac{\bar C}{4}.
\]
Clearly, up to a subsequence $\mf{u}_{R_n''} \to \mf{w}'' \in \mathcal{BD}_{\mf{u}}$ in $\mathcal{C}^0_{\loc}(\R^N)$, and
\[
\int_{\pa B_1} (w_i'')^2= \lim_{n \to \infty} \int_{\pa B_1} u_{i,R_n''}^2 = \frac{\bar C}{4}.
\]
This contradicts what we proved in the first step. 
\end{proof}

In light of Lemma \ref{lem: zero in all the blow down}, up to relabelling there exists $1 \le h \le k$ such that if $i =1,\dots,h$, then $u_{i,R_n} \not \to 0$ as $n \to +\infty$ for every $R_n \to +\infty$, while if $i=h+1,\dots,k$, then for the entire blow-down family $u_{i,R} \to 0$ as $R \to +\infty$ (in $\mathcal{C}^0_{\loc}(\R^N)$ and in $H^1_{\loc}(\R^N)$). For every $\eps,R>0$, we introduce
\begin{equation}\label{def D}
D_{\eps,R}=\left\{ x \in \pa B_1: \sum_{i=1}^h u_{i,R}^2(x)>\eps \right\},
\end{equation}
and its complement $D_{\eps,R}^c$.

\begin{lemma}\label{lem: misura a zero}
It results that 
\[
\lim_{\eps \to 0}\left(\lim_{R \to +\infty} \mathcal{H}^{N-1}(D_{\eps,R}^c) \right) = 0.
\]
\end{lemma}
\begin{proof}
By contradiction, there exists $\bar \delta>0$ and a sequence $\eps_n \to 0$ such that 
\[
\limsup_{R \to +\infty} \mathcal{H}^{N-1}(D_{\eps_n,R}^c) \ge \bar \delta \qquad \text{for every $n$}.
\]
Thus, we can find a sequence of real numbers $R_n >n$ such that 
\begin{equation}\label{abs2}
\mathcal{H}^{N-1}(D_{\eps_n,R_n}^c) > \frac{\bar \delta}{2} \qquad \text{for every $n$}.
\end{equation}
Up to a subsequence $\mf{u}_{R_n} \to \mf{u}_{\infty}$ in $\mathcal{C}^0_{\loc}(\R^N)$ and in $H^1_{\loc}(\R^N)$, as $n \to \infty$, and by Theorem \ref{them: blow down k dim 2} we know that $\mathcal{H}^{N-1}(\{\mf{u}_\infty=\mf{0}\} \cap \pa B_1)=0$. Since by Lemma \ref{lem: zero in all the blow down} we have also $\sum_{i=h+1}^k u_{i,\infty}^2= 0$ in $\pa B_1$, it is necessary that
\begin{equation}\label{** pag 3}
\mathcal{H}^{N-1}\left(\left\{ x \in \pa B_1: \sum_{i=1}^h u_{i,\infty}^2(x) = 0\right\} \right) = 0.
\end{equation}
We show that this leads to a contradiction with the estimate \eqref{abs2}. 
For any $\rho>0$, let 
\[
D_{\rho,\infty}:=\left\{x \in \pa B_1: \sum_{i=1}^h u_{i,\infty}^2(x) > \rho\right\},
\]
By \eqref{** pag 3} there exists $\bar \rho>0$ sufficiently small such that $\mathcal{H}^{N-1}(\pa B_1 \setminus \overline{D_{\bar \rho,\infty}}) \le\bar \delta/2$; by uniform convergence there exists $\bar n$ such that 
\[
\inf_{x \in  D_{\bar \rho,\infty}} \sum_{i=1}^h u_{i,R_n}^2(x) \ge \frac{\bar \rho}{2} \quad \text{and} \quad \eps_n<\frac{\bar \rho}{2}
\]
for every $n  \ge \bar n$. Therefore
\[
\mathcal{H}^{N-1}(D_{\eps_n,R_n}^c) \le \mathcal{H}^{N-1}\left(\left\{ x \in \pa B_1: \sum_{i=1}^h u_{i,R_n}^2(x) < \frac{\bar \rho}{2}\right\} \right) \le \mathcal{H}^{N-1}(\pa B_1 \setminus \overline{D_{\bar \rho,\infty}}) \le \frac{\bar \delta}{2}
\]
whenever $n>\bar n$, which is in contradiction with the \eqref{abs2}.
\end{proof}

Now the core of the proof of Proposition \ref{thm: zero total} begins. Let us assume by contradiction that there exists $\bar i=h+1,\dots,k$ such that, although $u_{\bar i,R} \to 0$ as $R \to +\infty$, it results $u_{\bar i} \not \equiv 0$. Without loss of generality, we suppose that $\bar i=h+1$. Clearly, 
\begin{equation}\label{sub h+1}
\begin{cases}
-\Delta u_{h+1} \le - u_{h+1} \sum_{j=1}^h u_j^2 & \text{in $\R^N$} \\
u_{h+1}>0 & \text{in $\R^N$}.
\end{cases}
\end{equation}
Let
\begin{equation}\label{def di f}
f(r):= \begin{cases} \frac{2-N}{2} r^2+ \frac{N}{2} & \text{if $0< r \le 1$} \\
r^{2-N} & \text{if $r>1$}.
\end{cases}
\end{equation}
Note that $f \in \mathcal{C}^1((0,+\infty))$ and $\Delta f(|x|) \le 0$ a.e. in $\R^N$. For $\beta>0$, let also
\begin{align*}
&\Lambda(r) := \frac{r^2 \int_{\pa B_r} |\nabla_\theta u_{h+1}|^2 +u_{h+1}^2 \sum_{j \neq i} u_j^2  }{\int_{\pa B_r} u_{h+1}^2} \\
& I_{\beta}(r) := \frac{1}{r^{\beta}} J(r) := \frac{1}{r_{\beta}}\int_{B_r} f(|x|) \left( |\nabla u_{h+1}|^2 + u_{h+1}^2 \sum_{j=1}^h  u_j^2 \right).
\end{align*}

\begin{lemma}\label{lem: step 1 monotonicity}
For every $r>1$, it holds
\[
J(r) \le \frac{r}{2\gamma(\Lambda(r))} \int_{\pa B_r} f(|x|) \left( |\nabla u_{h+1}|^2 + u_{h+1}^2 \sum_{j=1}^h u_j^2 \right),
\]
where we recall the definition of $\gamma$, see equation \eqref{def di gamma}.
\end{lemma}
We refer the reader to the proof of Lemma \ref{lem: step 1 monotonicity general}, which contains that of Lemma \ref{lem: step 1 monotonicity} as particular case. 

In the next lemma we show that the function $I_\beta$ is non-decreasing for sufficiently large radii.

\begin{lemma}\label{lem: monotonicity}
Let $\beta>2d$. There exists $r_\beta \gg 1$ sufficiently large such that the function $r \mapsto I_{\beta}(r)$ is monotone non-decreasing for $r>r_\beta$.
\end{lemma}
\begin{proof}
In what follows we consider scaled functions of type
\[
v_{j,r}(x) := \frac{u_{j}(rx)}{\left( \frac{1}{r^{N-1}} \int_{\pa B_r} u_{j}^2 \right)^{\frac{1}{2}} }  \qquad \text{for $r>0$}.
\]
Note that the $L^2$ norm of $v_{j,r}$ on $\pa B_1$ is normalized to $1$ for every $r$ and for every $j$, and moreover
\begin{equation}\label{bound v}
%\begin{split}
\int_{\pa B_1} |\nabla_\theta v_{h+1,r}|^2 \le \Lambda(r), 
%\int_{\pa B_1} v_{h+1,r}^2 \sum_{j=1}^h \|u_j(r\,\cdot)\|_{L^2(\pa B_1)}^2 v_{j,r}^2 & \le \frac{1}{r^2} \Lambda(r),
%\end{split}
\end{equation}
as one can easily check by direct computations. By Lemma \ref{lem: step 1 monotonicity} it results that
\[
\frac{I_{\beta}'(r)}{I_{\beta}(r)} = -\frac{\beta}{r} + \frac{\int_{\pa B_r} f(|x|) \left( |\nabla u_{h+1}|^2 + u_{h+1}^2 \sum_{j =1}^h u_j^2 \right)}{ \int_{B_r} f(|x|) \left( |\nabla u_{h+1}|^2 + u_{h+1}^2 \sum_{j = 1}^h u_j^2 \right)}  \ge -\frac{\beta}{r} + \frac{2 \gamma(\Lambda(r))}{r}
\]
for every $r>1$. Hence to complete the proof of the lemma we wish to show that there exists $r_\beta\gg 1$ such that if $r > r_\beta$, then $2 \gamma(\Lambda(r))> \beta$. Let us assume by contradiction that such a value $r_\beta$ does not exist: then we can find a sequence $R_n \to +\infty$ such that 
\begin{equation}\label{claim absurd}
2 \gamma(\Lambda(R_n)) \le \beta \qquad \text{for every $n$},
\end{equation}
and recalling the definition of $\gamma$, see \eqref{def di gamma}, this implies that $(\Lambda(R_n))_n$ is bounded. By \eqref{bound v}, we deduce that the sequence $\{v_{h+1,R_n}\}$ is bounded in $H^1(\pa B_1)$, so that u.t.s. it is convergent to some $v_{h+1} \in H^1(\pa B_1)$ weakly in $H^1(\pa B_1)$, strongly in $L^2(\pa B_1)$, and a.e. in $\pa B_1$; in particular, $\|v_{h+1}\|_{L^2(\pa B_1)}=1$. Let $\omega_{h+1}:= \supp v_{h+1}$. If we can prove that 
\begin{equation}\label{estimate on eigenvalue support limit}
\gamma(\lambda_1(\omega_{h+1}))>\beta
\end{equation}
(where we recall that $\lambda_1(\omega_{h+1})$ is the first eigenvalue of the Laplace-Beltrami operator with Dirichlet boundary condition on $\omega_{h+1}$), then we easily reach a contradiction: indeed by the monotonicity of $\gamma$, and by the inequalities \eqref{bound v} and \eqref{claim absurd}, we deduce that
\begin{multline*}
\beta <\gamma(\lambda_1(\omega_{h+1})) \le \gamma \left( \int_{\pa B_1} |\nabla_\theta v_{h+1}|^2 \right)  \le \gamma\left(\liminf_{n \to \infty}  \int_{\pa B_1} |\nabla_\theta v_{h+1, R_n}|^2 \right) \\
\le \gamma\left( \liminf_{n \to \infty} \Lambda(R_n) \right) = \liminf_{n \to \infty} \gamma\left(  \Lambda(R_n) \right)  \le \beta,
\end{multline*}
a contradiction. To prove the \eqref{estimate on eigenvalue support limit}, we recall that $\lambda(\omega) \to +\infty$ as $\mathcal{H}^{N-1}(\omega) \to 0$: this implies that there exists $\delta>0$ such that 
\[
\mathcal{H}^{N-1}(\omega_{h+1}) < \delta \quad \Longrightarrow \quad \gamma(\lambda_1(\omega_{h+1}))>\beta,
\]
where we used the coercivity of $\gamma$. Therefore, the desired result follows if we show that 
\begin{equation}\label{eq: new thesis}
\mathcal{H}^{N-1}(\omega_{h+1}) < \delta. 
\end{equation}
As a first step, we observe that for every $n$
\begin{equation}\label{eq for v}
-\Delta v_{h+1,R_n} \le - R_n^2 H(\mf{u},0,R_n) \sum_{j=1}^h u_{j,R_n}^2 v_{h+1,R_n} \qquad \text{in $\R^N$},
\end{equation}
where we recall that $u_{j,R_n}(x)= u_j(R_n x)/H(\mf{u},0,R_n)^{1/2}$ (we emphasize the fact that the normalization of $u_{j,R_n}$ is different with respect to that of $v_{h+1,R_n}$). Moreover, having assumed that $u_{h+1} \not \equiv 0$, by the mean value inequality and Proposition \ref{prop: Almgren crescita puntuale}
\begin{equation}\label{alg growth v}
v_{h+1,R_n}(x) = \frac{u_{h+1}(R_n x)}{\left(\frac{1}{R_n^{N-1}} \int_{\pa B_{R_n}} u_{h+1}^2 \right)^{1/2}} \le \frac{C(1+R_n^d|x|^d)}{u_{h+1}^2(0)} \le C(1+R_n^d |x|^d)
\end{equation}
for every $x \in \R^N$, where $C>0$ is independent of $n$.

Now, by Lemma \ref{lem: misura a zero} there exists $\bar \eps>0$ such that for every $0<\eps<\bar \eps$
\begin{equation}\label{24gen1}
\mathcal{H}^{N-1} \left( D_{2\eps,r}^c \right) <\delta \qquad \text{provided $r>r_\eps$},
\end{equation}
for some $r_\eps>0$ sufficiently large. For the reader's convenience, we recall that $D_{2\eps,r}$ has been defined in \eqref{def D}. Clearly, there exists $\bar n_\eps$ such that $R_n>r_\eps$ whenever $n>\bar n_\eps$.  \\
In what follows we fix $\eps \in (0,\bar \eps)$, and we consider the blow-down sequence $\{\mf{u}_{R_n}: n >\bar n_\eps\}$. By usual arguments it is uniformly bounded in the $1/2$-H\"older norm in compact sets, it is convergent to a limiting profile $\mf{u}_{\infty} \in \mathcal{BD}_{\mf{u}}$ in $\mathcal{C}^0_{\loc}(\R^N)$, and thanks to Lemma \ref{lem: zero in all the blow down} we know that $u_{j,\infty} \equiv 0$ in $\R^N$ for $j=h+1,\dots,k$. For any $\rho>0$, let
\[
D_{\rho,\infty}:=\left\{ x \in \pa B_1: \sum_{i=1}^h u_{i,\infty}^2(x)>\rho \right\}.
\]
% equivalently, $\mathcal{H}^{N-1}(D_{\eps,\infty}) > \mathcal{H}^{N-1}(\S^{N-1})-\delta$. 
By uniform convergence, there exists $\bar n_\eps'$ such that
\[
\inf_{x \in D_{\eps,\infty}} \sum_{j=1}^h u_{j,R_n}^2(x) > \frac{\eps}{2} \qquad \text{for every $n>\bar n_\eps'$};
 \]
furthermore, by the uniform $1/2$-H\"older regularity of the sequence $\{\mf{u}_{R_n}\}$, there exists $\rho_\eps>0$ independent of $x_0 \in D_{\eps,\infty}$ and of $n>\bar n_\eps'$ such that
\begin{equation}\label{24gen2}
\sum_{j=1}^h u_{j,R_n}^2(x) >\frac{\eps}{4} \qquad \text{for every $x \in B_{\rho_\eps}(x_0)$}.
\end{equation}

Collecting equations \eqref{eq for v}, \eqref{alg growth v} and \eqref{24gen2}, we deduce that for every $x_0 \in D_{\eps,\infty}$ and for every $n>\bar n_\eps'$ we have
\[
\begin{cases}
-\Delta v_{h+1,R_n}  \le - \frac{\eps}{4} R_n^2 H(\mf{u},0,R_n) v_{h+1,R_n} & \text{in $B_{\rho_\eps}(x_0)$} \\
v_{h+1,R_n} \ge 0 & \text{in $B_{\rho_\eps}(x_0)$} \\
v_{h+1,R_n} \le C(1+R_n^d) & \text{in $B_{\rho_\eps}(x_0)$}.
\end{cases}
\]
As a consequence we can apply Lemma \ref{lem: exp decay}:
\[
v_{h+1,R_n}(x_0) \le C(1+R_n^d)e^{-C \eps R_n H(\mf{u},0,R_n)^{1/2} \rho_\eps} \le C(1+R_n^d)e^{-C \eps R_n \rho_\eps}
\]
for every $x_0 \in D_{\eps,\infty}$, for every $n>\bar n_\eps'$; notice that $H(\mf{u},0,R_n) \ge C >0$ by Proposition \ref{prop: doubling}-($i$). Passing to the limit as $n \to \infty$, since $\eps$ is fixed we infer that $v_{h+1}(x_0)=0$ for every $x_0 \in D_{\eps,\infty}$, that is, $\supp(\omega_{h+1}) \subset D_{\eps,\infty}^c$. Thus, to complete the proof it remains to show that the measure of $D_{\eps,\infty}^c$ in $\pa B_1$ is sufficiently small. 
%Since $D_{\eps,\infty}^c$ is compact and $\mf{u}_{R_n} \to \mf{u}_\infty$ uniformly in $\overline{B_2}$, 
By uniform convergence there exists $\bar n_\eps''$ such that 
\[
\sup_{x \in D_{\eps,\infty}} \sum_{j=1}^h u_{j,R_n}^2(x) \le \frac{3\eps}{2} \qquad \text{ for every $n>\bar n_\eps''$};
\]
in other worlds, $D_{\eps,\infty}^c \subset D_{2\eps,R_n}^c$ for every $n>\bar n_{\eps}''$, and thanks to the estimate \eqref{24gen1} we deduce that $\mathcal{H}^{N-1}(D_{\eps,\infty}^c) < \delta$; in particular the \eqref{eq: new thesis} holds, proving the thesis.
\end{proof}

We are ready to complete the proof of Proposition \ref{thm: zero total}. The basic idea is that the monotonicity formula proved in the previous lemma imposes a minimal growth rate on the function $J$ for $r$ large, while the algebraic growth of $\mf{u}$ gives a maximal growth rate, and having chosen $\beta>2d$ these two estimates are in contradiction. This kind of argument is by now well understood (see for instance the proofs of Proposition 7.1 in \cite{CoTeVe} and of Proposition 2.6 in \cite{NoTaTeVe}), even though it is usually employed on groups of components rather than on a single one. This is related to the fact that in the present setting, using the assumption $u_{h+1,R} \to 0$ as $R \to +\infty$, we could prove that the ``asymptotic support" of $u_{h+1}$ (the set $\omega_{h+1}$ in the previous proof) is arbitrarily small, thus proving a one-phase monotonicity formula.

\begin{proof}[Conclusion of the proof of Proposition \ref{thm: zero total}]
By Lemma \ref{lem: monotonicity}, it results that
\begin{equation}\label{stima 1}
\int_{B_r} f(|x|)\left( |\nabla u_{h+1}|^2 + u_{h+1}^2 \sum_{j =1}^h u_j^2\right) \ge C r^\beta
\end{equation}
for every $r>r_\beta$. On the other hand, 
%let $\nu \in \mathcal{C}^\infty([0,+\infty))$ be such that $\neta(r) = 1$ for $r \in [0,1]$ and $\neta(r) = 0$; let $\neta_r(x):= \neta(|x|/r)$, in such a way that $\neta_r \equiv 1$ in is a smooth cut off function in $B_r$ such that 
let us test the inequality \eqref{sub h+1} against $\eta^2 f(|x|) u_{h+1}$, where $\eta$ is a smooth cut of function such that $\eta \equiv 1$ in $B_r$, $\eta \equiv 0$ in $\R^N \setminus B_{2r}$, and $|\nabla \eta| \le C/r$ in $\R^N$. By means of some integrations by parts, we obtain
\begin{multline*}
\int_{\R^N} \eta^2 f(|x|)\left( |\nabla u_{h+1}|^2 + u_{h+1}^2 \sum_{j =1}^h u_j^2\right) \\ \le - \int_{\R^N} \left[ 2u_{h+1} \eta f(|x|) \nabla u_{h+1} \cdot \nabla \eta + u_{h+1} \eta^2 \nabla u_{h+1} \cdot \nabla f(|x|) \right] \\
 \le \int_{\R^N} \left[ 2 f(|x|) u_{h+1}^2 |\nabla \eta|^2 + \frac{1}{2} f(|x|)\eta^2|\nabla u_{h+1}|^2 -  \eta^2 \nabla \left(\frac{u_{h+1}^2}{2} \right) \cdot \nabla f(|x|) \right] \\
 = \int_{\R^N} \left[ 2 f(|x|) u_{h+1}^2 |\nabla \eta|^2 + \frac{1}{2} f(|x|)\eta^2|\nabla u_{h+1}|^2 \right. \\
 \left.- \nabla \left(\frac{\eta^2 u_{h+1}^2}{2}\right) \cdot \nabla f(|x|)+ u_{h+1}^2 \eta \nabla \eta \cdot \nabla f(|x|) \right].
\end{multline*} 
Since $f(|x|)$ is superharmonic (recall the definition \eqref{def di f}), it results that
\[
-\int_{\R^N} \nabla \left(\eta^2 u_{h+1}^2\right) \cdot \nabla f(|x|) = \int_{\R^N} \eta^2 u_{h+1}^2  \Delta f(|x|) \le 0,
\]
and as a consequence
\[
\int_{\R^N} \eta^2 f(|x|)\left( |\nabla u_{h+1}|^2 + u_{h+1}^2 \sum_{j =1}^h u_j^2\right) \le 2 \int_{\R^N} \left[ 2 f(|x|) u_{h+1}^2 |\nabla \eta|^2 + u_{h+1}^2 \eta \nabla \eta \cdot \nabla f(|x|) \right].
\]
By the choice of $\eta$ and the definition of $f$ \eqref{def di f}, we infer
\begin{equation}\label{stima 2}
\begin{split}
\int_{B_r} f(|x|)\left( |\nabla u_{h+1}|^2 + u_{h+1}^2 \sum_{j =1}^h u_j^2\right) & \le \int_{B_{2r} \setminus B_r} \left[ 2 f(|x|) u_{h+1}^2 |\nabla \eta|^2 + u_{h+1}^2 \eta \nabla \eta \cdot \nabla f(|x|) \right] \\
& \le \frac{C}{r^2} \int_{B_{2r} \setminus B_r} \frac{u_{h+1}^2}{|x|^{N-2}} + \frac{C}{r} \int_{B_{2r} \setminus B_r} \frac{u_{h+1}^2}{|x|^{N-1}} \\
& \le \frac{C}{r^2} \int_r^{2r} \rho^{2d+1} \, \de\rho + \frac{C}{r} \int_r^{2r} \rho^{2d}\, \de\rho =C r^{2d},
\end{split}
\end{equation}
where we used the fact that $u_{h+1}(x) \le C(1+|x|^d)$ for every $x \in \R^N$. Having chosen $\beta>2d$, a comparison between \eqref{stima 1}  and \eqref{stima 2} gives a contradiction for $r$ sufficiently large.
\end{proof}  

\section{Liouville-type theorems for system \eqref{system}}\label{sec: Liouv 2}

The aim of this section is to prove Theorems \ref{thm: Liouville dim 2} and \ref{thm: higher dim}, and Corollary \ref{protothm 2}.

\subsection{$2$-dimensional case}

\begin{proof}[Proof of Theorem \ref{thm: Liouville dim 2}]
Let $\mf{u}$ be a positive solution of \eqref{system} such that  $N(\mf{u},0,+\infty) = d \in (0,+\infty)$, and let us consider the blow-down family $\{\mf{u}_{R}\}$. By Theorem \ref{them: blow down k dim 2}, up to a subsequence and up to a rotation 
\[
\mf{u}_R \to (\chi_{A_1},\dots,\chi_{A_k})\pi^{-1/2} r^d \sin(d\theta) \qquad \text{as $R \to +\infty$}
\]
uniformly on compact sets and in $H^1_{\loc}(\R^N)$. Moreover, $\chi_{A_i} \neq 0$ for every $i$, since otherwise by Proposition \ref{thm: zero total} we would have $u_i \equiv 0$ in $\R^N$, in contradiction with the positivity of the considered solution. Now, any $A_i$ is the union of non-adjacent nodal domains of the function $r^d \sin (d\theta)$; each nodal domain is a cone of angle $\pi/d$, so that we have exactly $2d$ nodal domains. Since for every $i$ the positivity domain $A_i$ contains at least one cone, we deduce that $k \le 2d$. 
\end{proof}

\subsection{Higher dimensional case}\label{sec: higher dimension}

We need a monotonicity result for the dependence of $\mathcal{L}_k(\S^{N-1})$ with respect to $k$.
%It remains to show that if $N(\mf{u},0,+\infty)=3/2$ then $\mf{u}$ has exactly $3$ positive components. For the proof,  

\begin{lemma}\label{lem: monot part}
For every $N,k \ge 2$, it results that $\mathcal{L}_{k+1}(\S^{N-1})>\mathcal{L}_k(\S^{N-1})$.
\end{lemma}
\begin{proof}
Let $\omega=(\omega_1,\dots,\omega_{k+1})$ be an optimal $(k+1)$-partition for $\mathcal{L}_{k+1}(\S^{N-1})$ (the existence of $\omega$ is given by Theorem 3.4 in \cite{HeHHTer1}). Up to a relabelling, we can assume that $\omega_1$ and $\omega_{k+1}$ are adjacent sets of the partition. We consider the connected set $\omega_1'=\Int(\overline{\omega_1 \cup \omega_{k+1}})$, and the $k$-partition $\omega'=(\omega_1',\omega_2,\dots,\omega_k)$ . Since clearly $\lambda_1(\omega_1') \le \lambda_1(\omega_1)$, we deduce that
\[
\max\{\lambda_1(\omega_1'),\lambda_1(\omega_2),\dots,\lambda_1(\omega_k)\} = \max_{i=1,\dots,k+1} \lambda_1(\omega_i) = \mathcal{L}_{k+1}(\S^{N-1}),
\]
which in turn yields $\mathcal{L}_{k+1}(\S^{N-1}) \ge \mathcal{L}_k(\S^{N-1})$. To show that the strict inequality holds, we assume by contradiction that the values are equal. Arguing as in the first part of the proof, from an optimal $(k+1)$-partition $\omega$ for $\mathcal{L}_{k+1}(\S^{N-1})$ we can construct a $k$-partition $\omega''=(\omega_1'',\omega_2,\dots,\omega_k)$ such that
\[
\max\{\lambda_1(\omega_1''),\lambda_1(\omega_2),\dots,\lambda_1(\omega_k)\} = \max_{i=1,\dots,k+1} \lambda_1(\omega_i) = \mathcal{L}_{k+1}(\S^{N-1})= \mathcal{L}_k(\S^{N-1}),
\]
that is $\omega''$ is optimal for $\mathcal{L}_k(\S^{N-1})$. But
$\lambda_1(\omega_1'') \neq \lambda_1(\omega_j)$ for every $j \neq 1$, in contradiction with the fact that for any optimal $k$-partition $\bar \omega$ it holds $\lambda_1(\bar \omega_i)= \lambda_1(\bar \omega_j)$ for every $i \neq j$, see Theorem 3.4 in \cite{HeHHTer1}.
%
% it is known that any optimal partition $\bar \omega$  for $\mathcal{L}_k(\S^{N-1})$ satisfies $\lambda_1(\bar \omega_i) = \lambda_1(\bar \omega_j)$ for every $i \neq j$, a contradiction.
%
%
% and consider again an optimal $(k+1)$-partition $\omega=(\omega_1,\dots,\omega_{k+1})$ for $\mathcal{L}_{k+1}(\S^{N-1})$. Up to a relabelling, we can assume that $\omega_1$ and $\omega_{2}$ are adjacent sets of the partition. Let $\omega_1''=\omega_1 \cup \omega_2$, and let us consider the $k$-partition $\omega''=(\omega_1'',\omega_3,\dots,\omega_{k+1})$. Since $\lambda_1(\omega_1'')<\lambda_1(\omega_3)$, we have
%\[
%\max\{\lambda_1(\omega_1''),\lambda_1(\omega_3),\dots,\lambda_k(\omega_{k+1})\} = \max_{i=1,\dots,k+1} \lambda_1(\omega_i) = \mathcal{L}_{k+1}(\S^{N-1}),
%\]
%and $\omega''$ is optimal for $\mathcal{L}_k(\S^{N-1})$. On the other hand, by Theorem 3.4 in \cite{HeHHTer1} it is known that any optimal partition $\bar \omega$  for $\mathcal{L}_k(\S^{N-1})$ satisfies $\lambda_1(\bar \omega_i) = \lambda_1(\bar \omega_j)$ for every $i \neq j$, a contradiction.
\end{proof}

\begin{proof}[Proof of Theorem \ref{thm: higher dim}]
We consider a positive solution $\mf{u}$ of \eqref{system}. Arguing as in the previous proof, u.t.s. the blow-down family is convergent to a limiting profile $r^d \mf{g}(\theta)$ uniformly on compact sets and in $H^1_{\loc}(\R^N)$, see Theorem \ref{them: blow down k dim 2}. We point out that thanks to Proposition \ref{thm: zero total} and having assumed that $\mf{u}$ is positive, $g_i \not \equiv 0$ in $\S^{N-1}$ for every $i$. Now, since $r^d g_i(\theta)$ is harmonic in its positivity domain, we have
\[
\begin{cases}
-\Delta_{\S^{N-1}} g_i = d(d+N-2) g_i & \text{in $\{g_i>0\}$}\\
g_i=0 & \text{on $\pa \{g_i>0\}$}.
\end{cases}
\]
Let $A_i:= \{g_i>0\}$, and let $\omega_{i,1},\dots,\omega_{i,l_i}$ be the nodal domains of $g_i$. We observe that since the functions $g_i$ have disjoint support
\[
\omega:=(\omega_{1,1},\dots,\omega_{1,l_1},\omega_{2,1},\dots,\omega_{2,l_2}, \dots, \omega_{k,1},\dots,\omega_{k,l_k}) = (\omega_i: 1 \le i \le \tilde k)
\] 
is a $\tilde k$-partition of $\S^{N-1}$ for some $\tilde k \ge k$, and is such that
\[
d(d+N-2) = \max_{i=1,\dots,\tilde k} \lambda_1(\omega_i) \ge \mathcal{L}_{\tilde k}(\S^{N-1}) \ge \mathcal{L}_k(\S^{N-1}),
\] 
where we used the definition \eqref{def di L_k} of $\mathcal{L}_k(\S^{N-1})$ and Lemma \ref{lem: monot part}. Recalling the definition \eqref{def di gamma} of $\gamma$, this is equivalent to $d \ge \gamma(\mathcal{L}_k(\S^{N-1}))$. Notice also that if the equality holds, then necessarily $\tilde k=k$, and $(\omega_1,\dots,\omega_k)$ is an optimal partition for $\mathcal{L}_k(\S^{N-1})$.
\end{proof}

\subsection{Proof of Corollary \ref{protothm 2}}\label{sub: 3 comp} 

For point ($i$), as already observed in the introduction (see Remark \ref{rem: su k=2}), we can either recall that $\gamma(\mathcal{L}_2(\S^{N-1}))=1$ for every $N$ and use Theorem \ref{thm: higher dim}, or apply Corollary \ref{corol general}, whose proof is independent of the argument we developed so far. Thus in what follows we focus on the jump in the admissible values of $N(\mf{u},0,+\infty)$ and on point ($ii$) of the thesis.

\begin{lemma}\label{lem: 26 marzo}
Let $\mf{v} \in \mathcal{G}_{\loc}^*(\R^N)$ be homogeneous with respect to $0$. Then
\[
\text{either} \quad \tilde N(\mf{v},0,0^+) = 1, \quad \text{or} \quad  \tilde N(\mf{v},0,0^+) \ge \frac{3}{2}.
\]
\end{lemma}
For the reader's convenience, we recall the definitions of $\tilde N(\mf{v},0,r)$ and of and $\mathcal{G}_{\loc}^*(\R^N)$, see \eqref{def N segregated} and Definition \ref{def: G *}.
\begin{proof}
By homogeneity $0 \in \{\mf{v}=\mf{0}\}$, and hence by Proposition 4.2 in \cite{TaTe} the assertion holds true in dimension $N=2$. Now we show how to exploit a blow-up analysis in order to lower the dimension: we proceed by induction assuming that the result holds in dimension $N-1$, and proving that then it holds in dimension $N \ge 3$. Let $\mf{v}= r^d \mf{g}(\theta) \in \mathcal{G}_{\loc}^*(\R^{N})$, and let $\Gamma_{\mf{v}}:= \{\mf{v}=\mf{0}\} \cap \S^{N-1}$.

\emph{Case 1) $\tilde N(\mf{v},x_0,0^+)=1$ for every $x_0 \in \Gamma_{\mf{v}}$.} By the results in \cite{TaTe}, and observing that by homogeneity of $\mf{v}$ all the components of $\{\mf{v}=\mf{0}\}$ intersect $\S^{N-1}$ transversally, this implies that 
\begin{equation}\label{character gamma v}
\begin{split}
&\text{$\Gamma_{\mf{v}}$ is the union of disjoint compact $(N-2)$-dimensional}\\
&\text{hyper-surfaces of class $\mathcal{C}^{1,\alpha}$ without boundary.}
\end{split}
\end{equation} 
Let us denote by $\varphi_i$, $i=1,\dots$ the connected components of $\Gamma_{\mf{v}}$, by $\omega_i$, $i=1,\dots$ the connected components of $\S^{N-1} \setminus \Gamma_{\mf{v}}$, and by $\hat \omega_i \subset \R^N$ the cone projecting onto $\omega_i$. Let $\omega_i$ and $\omega_j$ be such that $\hat\omega_i \subset \supp\{v_h\}$, $\hat \omega_j \subset \supp\{v_k\}$ for some $h,k$. If $\omega_i$ and $\omega_j$ are adjacent, then since $\mf{v} \in \mathcal{G}_{\loc}^*(\R^N)$ we know that 
\[
-\Delta(v_h-v_k) = 0 \quad \text{in $\Int\left(\overline{ \hat{\omega_i} \cup \hat \omega_j}\right)$}.
\]
Thus, if we can apply this argument for every pair of adjacent connected components of $\S^{N-1} \setminus \Gamma_{\mf{v}}$, alternating the sign of $v_h$ and $v_k$ without any contradiction (it is not admissible to obtain two adjacent $+$ or two adjacent $-$), we can construct a spherical harmonic $\Psi_d$ such that such that
\[
\mf{v}(r,\theta)=\left(\chi_{A_1},\dots,\chi_{A_k}\right) r^d \Psi_d(\theta),
\]
where $d$ is a suitable integer, $(A_1,\dots,A_k)$ is a partition of $\Sigma_{\Psi_d}=\{\Psi_d \ne 0\}$, and $A_i$ is the union of non-adjacent nodal domains of $\Sigma_{\Psi_d}$. If this is possible, the thesis follows easily, as we deduce that the Almgren frequency function $\tilde N(\mf{v},0,r)$ is on one side constant and equal to $d$ by homogeneity (see Proposition \ref{prop: monot segregated}), while, on the other side, it is a positive integer by the harmonicity of $r^d \Psi_d$. So, if it is different from $1$, it has to be larger than or equal to $2 > 3/2$.

To prove that the definition of the spherical harmonic $\Psi_d$ is possible, we construct a graph $\mathcal{A}$ as follows: each vertex of $\mathcal{A}$ represents one component $\omega_i$, and two vertexes $\omega_i$ and $\omega_j$ are connected by an edge if they are adjacent. What in principle could prevent the possibility of constructing $\Psi_d$ is the occurrence of a loop in this graph. 

In the present setting \eqref{character gamma v} holds true, and by a generalization of the Jordan curve theorem this implies that each $\gamma_i$ separates $\S^{N-1}$ in exactly two connected components, say \emph{north} and \emph{south}. Notice that, being $\gamma_i$ regular, $\gamma_i \cap \overline{\omega_k} \neq \emptyset$ for exactly two indexes $k$. Therefore, any components in north cannot be adjacent to any component in south but at most one, and hence if a component $\omega_i$ in south is adjacent to a component $\omega_j$ in north, then $\omega_i$ cannot be adjacent to any other component $\omega_k$ which is adjacent to $\omega_j$. Since this holds for each $\gamma_i$, we deduce that there are no loops in $\mathcal{A}$, and hence the construction of $\Psi_d$ is possible. As observed, this proves the desired result in the considered situation.

\emph{Case 2) There exists $x_0 \in \Gamma_{\mf{v}}$ such that $\tilde N(\mf{v},x_0,0^+) >1$.} We consider a blow-up in a neighbourhood of $x_0$ by introducing, for any $\rho>0$,
\[
v_{i,\rho}(x):= \frac{v_{i}(x_0+ \rho x)}{H(\mf{v},x_0,\rho)^{1/2}}. 
\] 
Since $\mf{v} \in \mathcal{G}_{\loc}^*(\R^N)$, we are in position to apply Theorem 3.3 and Corollary 3.12 in \cite{TaTe}: there exists a sequence $\rho_m \to 0$ such that $\mf{v}_{m}:=\mf{v}_{\rho_m} \to \bar{\mf{v}}$ as $m \to +\infty$ in $\mathcal{C}^{0,\alpha}_{\loc}(\R^N)$ for every $0<\alpha<1$ and strongly in $H^1_{\loc}(\R^N)$. Furthermore, also $\bar{\mf{v}}(r,\theta)=r^{\sigma} \mf{g}(\theta)$ belongs to the class $\mathcal{G}_{\loc}^*(\R^N)$. Here $\sigma= \tilde N(\mf{v},x_0,0^+)>1$ by assumption. It is now crucial to observe that, thanks to the homogeneity of $\mf{v}$ with respect to $0$, the function $\bar{\mf{v}}$ depends only on $N-1$ variables. To be precise, we claim that
\begin{equation}\label{v n-1}
\bar{\mf{v}}(x+\lambda x_0)=\bar{\mf{v}}(x) \qquad \text{for every $\lambda>0$ and $x \in \R^N$}.
\end{equation}
Since $\mf{v}_{m} \to \bar{\mf{v}}$ in $\mathcal{C}^{0,\alpha}_{\loc}(\R^N)$, it is sufficient to show that $\mf{v}_{m}(x+\lambda x_0)-\mf{v}_{m}(x) \to 0$ as $m \to \infty$ for any $\lambda>0$ and $x \in \R^N$. Let $q$ be the degree of homogeneity of $\mf{v}$ with respect to $0$ ($q$ can be different from $\sigma$). We have
\begin{multline*}
\mf{v}_{m}(x+\lambda x_0)  = \frac{\mf{v}(x_0+\rho_m(x+\lambda x_0)) }{H(\mf{v},x_0,\rho_m)^{1/2}}  = \frac{\mf{v}((1+\lambda \rho_m) x_0+\rho_m x) }{H(\mf{v},x_0,\rho_m)^{1/2}} \\
 = \frac{(1+\lambda \rho_m)^{q}}{H(\mf{v},x_0,\rho_m)^{1/2}} \mf{v}\left(x_0+\frac{\rho_m}{1+\lambda \rho_m} x \right) = (1+\lambda \rho_m)^{q} \mf{v}_m\left( \frac{x}{1+\lambda \rho_m} \right).
\end{multline*}
Let $K$ be a compact set containing both $x$ and $x/(1+\lambda \rho_m)$ for $m$ sufficiently large. Then the local $\mathcal{C}^{0,\alpha}$ convergence of $\mf{v}_m$ to $\bar{\mf{v}}$ implies that
\begin{multline*}
|\mf{v}_{m}(x+\lambda x_0)-\mf{v}_{m}(x)| \le \left| (1+\lambda \rho_m)^q \mf{v}_{m}\left( \frac{x}{1+\lambda \rho_m} \right)- \mf{v}_{m}\left( \frac{x}{1+\lambda \rho_m} \right)\right| \\
 + \left| \mf{v}_{m}\left( \frac{x}{1+\lambda \rho_m} \right)-\mf{v}_{m}(x) \right|
 \le C \left|(1+\lambda \rho_m)^{q}-1\right| + C \left| \frac{1}{1+\lambda \rho_m}-1\right|^{\alpha} |x|^{\alpha},  
\end{multline*}  
from which the claim \eqref{v n-1} follows. Now, up to a rotation we can assume that $x_0=(0,\dots,0,1)$, so that by \eqref{v n-1} we infer that $\bar{\mf{v}}$ is independent on $x_N$, and we recall that $\bar{\mf{v}} \in \mathcal{G}_{\loc}^*(\R^N)$ is homogeneous with homogeneity degree $\sigma>1$. It is then clear that the restriction $\bar{\mf{v}}|_{\R^{N-1} \times \{0\}}$ belongs to $\mathcal{G}_{\loc}^*(\R^{N-1})$ with the same degree $\sigma$, so that by inductive assumption and by homogeneity we have
\begin{align*}
1< \sigma &=\tilde N(\mf{v},x_0,0^+) = \lim_{m \to \infty} \tilde N(\mf{v},x_0,\rho_m) = \lim_{m \to \infty} \tilde N(\mf{v}_m,0,1) \\
& = \tilde N(\bar{\mf{v}},0,1) = \tilde N(\bar{\mf{v}}|_{\R^{N-1} \times \{0\}},0,1)  = \tilde N(\bar{\mf{v}}|_{\R^{N-1} \times \{0\}},0,0^+) 
\end{align*}
has to be larger than $3/2$. Using another time Proposition \ref{prop: monot segregated} and Lemma \ref{lem: constancy of N segregated}, we conclude that
\[
\frac{3}{2} \le \sigma = \tilde N(\mf{v},x_0,0^+) \le \tilde N(\mf{v},x_0,+\infty) = \tilde N(\mf{v},0,+\infty) = \tilde N(\mf{v},0,0^+),
\]
which completes the proof.
\end{proof}

\begin{remark}
In \cite{TaTe} it has been proved that if $\mf{v} \in \mathcal{G}_{\loc}(\R^N)$ then either $\tilde N(\mf{v},0,0^+) = 1$ or $\tilde N(\mf{v},0,0^+) \ge 1+ \delta_N$ for some $\delta_N>0$ depending only on the dimension $N$. The extra information which we obtain in Lemma \ref{lem: 26 marzo} is the precise value of $\delta_N$ (which is remarkably independent of $N$) under the extra-homogeneity assumption.
\end{remark}

\begin{remark}
The quoted generalization of the Jordan curve theorem is false when $N=2$, as in such case $(N-2)$ dimensional sets are points, and a point does not divide $\S^1$ in two connected components. There is no contradiction with our argument, since the case $N=2$ follows from \cite{TaTe}, where it is treated with a different method. 
\end{remark}

We are now in position to compute the optimal value $\mathcal{L}_3(\S^{N-1})$ in any dimension $N \ge 3$. 

\begin{proof}[Proof of Theorem \ref{thm: on optimal partition}]
We aim at proving that, for every $N \ge 3$, there holds 
\[
\frac{3}{2}\left(\frac{3}{2}+ N-2\right) = \mathcal{L}_3(\S^{N-1}). 
\]
By definition it is equivalent to show that $3/2 =\gamma(\mathcal{L}_3(\S^{N-1}))$. The identity is satisfied when $N=3$ in light of Theorem 1.1 in \cite{HeHHTer2}. We observe that the optimal partition in dimension $3$ provides an admissible partition in any dimension, so that $\gamma(\mathcal{L}_3(\S^{N-1}))\le 3/2$ for every $N$. By Theorem 3.4 in \cite{HeHHTer1} there exists an optimal partition $\omega \in \mathcal{P}_3(\S^{N-1})$ achieving $\mathcal{L}_3(\S^{N-1})$, and such that $\lambda_1(\omega_i)=\lambda_1(\omega_j)$ for every $i \neq j$; moreover, there exists eigenfunctions $\{\varphi_i\}$ corresponding to $\{\lambda_1(\omega_i)\}$ such that 
\[
-\Delta_{\S^{N-1}} \left(\varphi_i-\varphi_j\right) = \mathcal{L}_3(\S^{N-1}) \left(\varphi_i-\varphi_j\right) \qquad \text{in $\Int(\overline{\omega_{i,k} \cup \omega_{j,h}})$}
\]
where $\omega_{i,k}$ and $\omega_{j,k}$ are adjacent connected components of $\omega_i$ and $\omega_j$, respectively. Let $\mf{v}:=r^{\gamma(\mathcal{L}_3(\S^{N-1}))}(\varphi_1,\dots,\varphi_k)$. As observed in Section 8 in \cite{TaTe}, $\mf{v} \in \mathcal{G}_{\loc}(\R^N)$, and by homogeneity we can say even more: $\mf{v} \in \mathcal{G}_{\loc}^*(\R^N)$ and is homogeneous with respect to $0$. Therefore by Lemma \ref{lem: 26 marzo} 
\[
\text{either} \quad \gamma(\mathcal{L}_3(\S^{N-1})) = \tilde N(\mf{v},0,0^+)=1, \quad \text{or} \quad \gamma(\mathcal{L}_3(\S^{N-1})) = \tilde N(\mf{v},0,0^+) \ge \frac{3}{2}.
\]
To obtain the desired result we have to rule out the former alternative. If $N(\mf{v},0,0^+)=1$, then by Lemma 6.1 in \cite{TaTe} the nodal set $\{\mf{v}=\mf{0}\}$ is a hyper-plane, which implies that only $2$ sets $\omega_i$ are not empty. But it is not difficult to see that any partition of type $\omega= (\omega_1,\omega_2,\emptyset)$ cannot be optimal for $\mathcal{L}_3(\S^{N-1})$. Indeed, in such a situation we could replace $\omega$ by $\omega'$ obtained after a splitting of $\omega_2$ in two non-empty sets. This would be another partition achieving $\mathcal{L}_3(\S^{N-1})$, but such that $\lambda_1(\omega'_1) \neq \lambda_1(\omega_2')$, in contradiction with Theorem 3.4 in \cite{HeHHTer1}. 
\end{proof}

\begin{proof}[Proof of point ($ii$) in Corollary \ref{protothm 2}]
%We start showing that if $\mf{u}$ is a nonnegative solution of \eqref{system} such that $N(\mf{u},0,+\infty) =: d \in (0,+\infty)$, then either $d=1$ or $d \ge 3/2$. 
If $\mf{u}$ has exactly $2$ non-trivial components, then $d \in \N$ by Theorem \ref{thm blow down 2 comp}. If $\mf{u}$ has $k \ge 3$ non-trivial components, then by Theorems \ref{thm: higher dim} and \ref{thm: on optimal partition} and Lemma \ref{lem: monot part} it is necessary that 
\begin{equation}\label{eq 12 aprile}
d \ge \gamma(\mathcal{L}_k(\S^{N-1})) \ge \gamma(\mathcal{L}_3(\S^{N-1})=3/2,
\end{equation}
where we used the monotonicity of $\gamma$ (see \eqref{def di gamma}). In any case, if $d > 1$, then $d \ge 3/2$. Finally, by Lemma \ref{lem: monot part} the inequality \eqref{eq 12 aprile} is strict whenever $k>3$, so that if $d=3/2$, then $k=3$.
\end{proof}

\section{Liouville-type theorems for a general class of systems}\label{sec: general}

This section is devoted to the proof of Theorem \ref{thm: general systems}.
In what follows $\mf{u}$ is a subsolution of \eqref{general syst}, and ($H1$)-($H3$) hold. We use the following notation:
\begin{align*}
&\Lambda_i(r) := \frac{r^2 \int_{\pa B_r} |\nabla_\theta u_{i}|^2 +u_{i}^2 g_i\left(x,\mf{u}\right)}{\int_{\pa B_r} u_{i}^2} \\
& J_{i}(r) := \int_{B_r} f(|x|) \left( |\nabla u_{i}|^2 + u_{i}^2 g_i(x,\mf{u}) \right).
\end{align*}

\begin{lemma}\label{lem: step 1 monotonicity general}
For every $r>1$ and $i=1,\dots,k$, it results
\[
J_i(r) \le \frac{r}{2\gamma(\Lambda_i(r))} \int_{\pa B_r} f(|x|) \left( |\nabla u_i|^2 + u_i^2 g_i(x,\mf{u}) \right).
\]
\end{lemma}
\begin{proof}
It is essentially contained in the proof of Lemma 7.3 in \cite{CoTeVe}, or in the proof of Lemma 2.5 in \cite{NoTaTeVe}. We report the sketch for the sake of completeness. \\
By testing the $i$-th inequality \eqref{general syst} against $f(|x|)u_i$ in $B_r$ with $r>1$ and after some computations, we obtain
\[
J_i(r) \le \frac{1}{r^{N-2}} \int_{\pa B_r} u_i \pa_{\nu} u_i+ \frac{N-2}{2r^{N-1}} \int_{\pa B_r} u_i^2,
\]
where $\pa_\nu$ denotes the outer normal derivative, as usual. By Young's inequality, it holds
\[
\left| \int_{\pa B_r} u_i \pa_{\nu} u_i\right| \le \frac{\gamma(\Lambda_i(r))}{2r} \int_{\pa B_r} u_i^2 + \frac{r}{2\gamma(\Lambda_i(r))} \int_{\pa B_r} (\pa_\nu u_i)^2.
\]
Hence, using the definition of $\gamma$, we deduce that
\begin{align*}
%\begin{multline*}
J_i(r) & \le \frac{1}{2r^{N-1} \gamma(\Lambda_i(r))} \left[ \left( \gamma(\Lambda_i(r))^2 + (N-2) \gamma(\Lambda_i(r)) \right) \int_{\pa B_r} u_i^2 + r^2 \int_{\pa B_r} (\pa_\nu u_i)^2 \right] \\
& = \frac{1}{2r^{N-3} \gamma(\Lambda_i(r))} \int_{\pa B_r} |\nabla_\theta u_i|^2 +u_i^2 g_i(x,\mf{u}) + (\pa_\nu u_i)^2,
%  \\
%&= \frac{r}{2\gamma(\Lambda_i(r))} \int_{\pa B_r} f(|x|) \left( |\nabla u_i|^2 + u_i g_i(x,\mf{u}) \right). \qedhere
%\end{multline*}
\end{align*} 
and the thesis follows.
\end{proof}

The following Alt-Caffarelli-Friedman monotonicity formula is the natural counterpart of Lemma 7.3 of \cite{CoTeVe} in the present setting.

\begin{lemma}\label{lem: monot several k}
Let us assume that for some $1 \le h \le k$ it holds $u_i>0$ in $\R^N$ for every $i=1,\dots,h$. Let $0<q<\beta(h,N)$, and let us consider the function
\[
J(r):= \prod_{i=1}^h \frac{1}{r^{q}} J_i(r).
\]
There exists $r'=r'(q)>0$ such that $J$ is monotone non-decreasing in $(r',+\infty)$.
\end{lemma}
\begin{remark}\label{rem: su beta relaxed}
The optimal value $\beta(k,N)$ defined in \eqref{def di beta} can be characterized as
\begin{equation}\label{caracter beta}
\beta(k,N) = \inf\left\{ \frac{2}{k}  \sum_{i=1}^k \gamma\left(\int_{\S^{N-1}} |\nabla_{\theta} u_i|^2\right)\left| \begin{array}{l}
u_i \in H^1(\S^{N-1}), \int_{\S^{N-1}} u_i^2 = 1, \\ u_i\cdot u_j \equiv 0 \ \forall i \neq j \end{array} \right.
\right\},
\end{equation}
see Section 2 of \cite{CoTeVePDE}.
\end{remark}
\begin{proof}
By Lemma \ref{lem: step 1 monotonicity}
\[
\frac{d}{dr} J(r) = -\frac{qh}{r} + \sum_{i=1}^h \frac{\int_{\pa B_r} f(|x|)\left( |\nabla u_i|^2 + u_i^2 g_i(x,\mf{u}) \right)   }{\int_{B_r} f(|x|)\left( |\nabla u_i|^2 + u_i^2 g_i(x,\mf{u}) \right)} \ge -\frac{qh}{r} + \sum_{i=1}^h \frac{2\gamma(\Lambda_i(r))}{r}.
\]
Therefore, we aim at proving that there exists $r'>1$ such that 
\[
\sum_{i=1}^h 2 \gamma(\Lambda_i(r))-hq \ge 0 \qquad \text{for every $r>r'$}.
\]
By contradiction, if this is not true there exists $r_n \to +\infty$ such that the left hand side is smaller than or equal to $0$, and by the monotonicity of $\gamma$ this implies that $(\Lambda_i(r_n))$ is bounded. Let us define
\[
u_i^{(r_n)}(x):= \frac{u_i(r_n x)}{\left(\frac{1}{r_n^{N-1}} \int_{\pa B_{r_n}} u_i^2 \right)^{1/2}}.
\]
One can easily compute
\begin{equation}\label{limitatezze acf generica}
\begin{split}
\int_{\pa B_1} |\nabla_\theta u_i^{(r_n)}|^2 & \le \Lambda_i(r_n)\\
\int_{\pa B_1} \left(u_i^{(r_n)}\right)^2 g_i\left(r_n x, \mf{u}(r_n x)\right) & \le \frac{\Lambda_i(r_n)}{r_n^2}.
\end{split}
\end{equation}
By the first one $\mf{u}^{(r_n)} \wc \tilde{\mf{u}}$ weakly in $H^1(\pa B_1)$, strongly in $L^2(\pa B_1)$, and a.e. in $\pa B_1$. We claim that $\tilde{\mf{u}}$ is segregated, that is $\tilde u_i \tilde u_j =0$ a.e. in $\pa B_1$ for every $i \neq j$. To check this, we note that by subharmonicity (recall ($H1$)) and by the fact that $u_1,\dots,u_h>0$ in $\R^N$ we have
\begin{equation}\label{eq aprile}
\frac{1}{r_n^{N-1}} \int_{\pa B_{r_n}} u_i^2 \ge \mathcal{H}^{N-1}(\S^{N-1}) u_i^2(0) \ge C_0 
\end{equation}
for every $i=1,\dots,h$ and for every $n$. Thus by Fatou's lemma
%, using assumptions ($H1$) and ($H3$), the second estimate in \eqref{limitatezze acf generica} and the Fatou lemma, we deduce that for every $i$
\begin{align*}
\int_{\pa B_1} \tilde u_i^2 \underline{g}_i\left(C_0^{1/2} \tilde{\mf{u}}\right) &\le \liminf_{n \to +\infty}  \int_{\pa B_1}\left(u_i^{(r_n)}\right)^2  \underline{g}_i\left(C_0^{1/2} \mf{u}^{(r_n)}\right) 
 \\
& \le (H1) \le  \liminf_{n \to +\infty}  \int_{\pa B_1}\left(u_i^{(r_n)}\right)^2 g_i\left(r_n x,C_0^{1/2} \mf{u}^{(r_n)}\right) \\
& \le ((H3)+ \eqref{eq aprile}) \le  \liminf_{n \to +\infty}  \int_{\pa B_1} \left(u_i^{(r_n)}\right)^2 g_i\left(r_n x,\mf{u}(r_n x)\right)= \eqref{limitatezze acf generica} = 0,
\end{align*}
which in light of assumption ($H2$) implies $\tilde u_j \tilde u_i = 0$ a.e. in $\pa B_1$, proving our claim. Now it is not difficult to conclude, by using the absurd assumption, the first estimate in \eqref{limitatezze acf generica}, the definition of $\gamma$ and the characterization of $\beta(h,N)$, see the \eqref{caracter beta}:
\[
\beta(h,N) h > q h \ge \liminf_{n \to +\infty} \sum_{i=1}^h 2 \gamma(\Lambda_i(r_n)) \ge  \sum_{i=1}^h 2 \gamma\left( \int_{\pa B_1} |\nabla_\theta \tilde u_i|^2  \right)  \ge  \beta(h,N) h,
\]
a contradiction. 
\end{proof}

%\begin{theorem}
%Let $N \ge 2$, and let $\mf{u}$ be a positive solution of \eqref{system} with $k\ge 2$ such that $N(r) \to d$ as $r \to +\infty$. Then $\beta(k,N) \le d$.\\
%In other words, let $\mf{u}$ be a non-trivial nonnegative solution of \eqref{system} in $\R^N$ such that $N(r) \to d$ as $r \to +\infty$, for some $d \in \N/2$; if $m$ denotes the largest positive number such $\beta(m,N) \le d$, then at most $m$ components do not vanish identically.
%\end{theorem}

\begin{proof}[Proof of Theorem \ref{thm: general systems}]
Let us assume by contradiction that for some $1 \le h \le k$ such that $\beta(h,N) > 2d$ there exists a solution of \eqref{system} with at least $h$ positive components, say $u_i>0$ for $i=1,\dots,h$. Let $q \in (2d,\beta(h,N))$. On one side, by Lemma \ref{lem: monotonicity} it results that 
\begin{equation}\label{eq1Liouv2}
\prod_{i=1}^h \frac{1}{r^{q}} J_i(r) \ge C \qquad \forall r \gg 1.
\end{equation}
On the other side, as in the estimate \eqref{stima 2} it results that that for every $i=1,\dots,h$ and $r>1$
\begin{equation}\label{eq2Liouv2}
J_i(r) \le C r^{2d},
\end{equation}
where we used the growth assumption on $\mf{u}$. Comparing \eqref{eq1Liouv2} and \eqref{eq2Liouv2}, and recalling that $q>2d$, we obtain a contradiction for $r$ sufficiently large.
\end{proof}

\section{Symmetry results}\label{sec: 1-dim}

We now pass to the $1$-dimensional symmetry of solutions to \eqref{system}. Thanks to Corollary \ref{protothm 2} we are in position to extend the main results in \cite{FaSo} and \cite{Wa,Wa2} for systems with an arbitrary number of components.

%\begin{theorem}
%Let $\mf{u}$ be a solution of \eqref{system} such that $N(r) \to d \in \R$ as $r \to +\infty$.
%\begin{itemize}
%\item[($i$)] If $d=1$, that is, if $\mf{u}$ has at most linear growth, then all the components $u_i$ but two, say $u_1$ and $u_2$, are identically zero, and $(u_1,u_2)$ is a $1$-dimensional solution of \eqref{system 2 comp}. 
%\item[($ii$)] If for some $i \neq j$
%\[
%\lim_{x_N \to \pm \infty} \left(u_i(x',x_N) - u_j(x',x_N) \right) = \pm \infty, 
%\]
%the limit being uniform in $x' \in \R^{N-1}$, then $d=1$, all the components $u_l$ with $l \neq i,j$ are identically zero, and $(u_i,u_j)$ is a $1$-dimensional solution of \eqref{system 2 comp}.
%\item[($iii$)] If the dimension $N=2$, and for some $i \neq j$
%\[
%\pa_N u_i >0 \quad \text{and} \quad \pa_N u_j<0 \quad \text{in $\R^N$},
%\]
%then $d=1$, all the components $u_k$ with $k \neq i,j$ are identically zero, and $(u_i,u_j)$ is a $1$-dimensional solution of \eqref{system 2 comp}.
%\end{itemize}
%\end{theorem}
\begin{proof}[Proof of Theorem \ref{thm: 1-dimensional symmetry}]
($i$) It is a straightforward consequence of point ($i$) in Corollary \ref{protothm 2} and of the main results in \cite{Wa} and \cite{Wa2}. \\
($ii$) Only to fix our minds, let $i=1$ and $j=2$. We know that $u_1(x',x_N) \to +\infty$ as $x_N \to +\infty$ uniformly in $x' \in \R^{N-1}$. Arguing as in the proof of Corollary 1.2 in \cite{FaSo}, we wish to show that as a consequence $u_{l,R}(x) \to 0$ as $R \to +\infty$ in the half-space $\R^{N-1}_+:=\R^{N-1} \times (0,+\infty)$ for every $l \neq 1$. Given $K>0$, by assumption there exists $M>0$ such that $u_1>K$ in $\{x_N>M/2\}$. For an arbitrary $\theta>1$, if $x \in \left\{x_N> M, \ |x'| < \theta x_N\right\}$ the ball $B_x := B_{x_{N}/100}(x)$ is contained in $\left\{x_N>M/2, \ |x'|< 2\theta x_N \right\}$. Consequently, if $x \in \left\{x_N>M, \ |x'| < \theta x_N\right\}$ we have 
\[
u_1(y) \ge K_x:= \inf_{z \in B_x} u(z) \ge K \qquad \forall y \in B_x, 
\]
and since $\mf{u}$ has algebraic growth
\[
u_l(y) \le C (1+|y|^d) \le C\left(1+ C(2 \theta +1)^d y_N^d) \right) \le C(1+ x_N^d)=:\delta_x,
\]
for every $y \in B_x$, for every $l \neq 1$. 
%The latter estimate implies that $\delta_x:=\sup_l \sup_{B_x} u_l \le C (1+x_N^p)$. 
Now,
\[
\begin{cases}
-\Delta u_l \le -K^2 u_l & \text{in $B_x$} \\
u_l \ge 0 & \text{in $B_x$}  \\
u_l \le \delta_x & \text{in $B_x$},
\end{cases}
\]
so that Lemma \ref{lem: exp decay} applies: 
\begin{equation}\label{eq dicembre}
u_l(x) \le C \delta_x e^{-C K x_N} \le C (1+x_N^d) e^{-C K x_N} \qquad \forall x \in \left\{x_N>M, \ |x'| < \theta x_N\right\},
\end{equation}
for every $l \neq 1$. Let $x \in \left\{x_N>0, |x'| < \theta x_N\right\}$; there exists $R_x>0$ such that $Rx \in  \left\{x_N>M, \ |x'| < \theta x_N\right\}$ for every $R>R_x$. Therefore
\[
\lim_{R \to +\infty} u_{l,R}(x) = \lim_{R \to +\infty} \frac{u_l(Rx)}{\sqrt{H(\mf{u},0,R)}}  = 0 \qquad \forall x \in \left\{x_N>0, |x'| < \theta x_N\right\},
\]
where, in addition to \eqref{eq dicembre}, we used that $H(\mf{u},0,R_n) \ge C >0$ by Proposition \ref{prop: doubling}-($i$). As $\theta$ has been arbitrarily chosen, we deduce that $u_{l,R} \to 0$ pointwise in $\R^N_+$ for every $l \neq 1$. On the other hand, by Theorem \ref{them: blow down k dim 2}, we know that up to a subsequence $\mf{u}_R \to \mf{u}_\infty$ in $\mathcal{C}^0_{\loc}(\R^N)$, where $\mf{u}_\infty$ has the properties described in Theorem \ref{them: blow down k dim 2}. We infer that $u_{l,\infty}=0$ in $\R^N_+$ for every $l \neq 1$. Analogously, starting from the fact that $u_2(x',x_N) \to +\infty$ as $x_N \to -\infty$ uniformly in $x' \in \R^{N-1}$, we deduce that $u_{l,\infty} \equiv 0$ in $\R^N_-:=\R^{N-1} \times (-\infty,0)$ for every $l \neq 2$. Since $u_{l,\infty}$ is continuous, $u_{l,\infty} \equiv 0$ in $\R^N$ for every $l \neq 1,2$, and thanks to Proposition \ref{thm: zero total} this implies that $u_l \equiv 0$ in $\R^N$ for any such $l$. Therefore $(u_1,u_2)$ is a solution of the $2$-component system \eqref{system 2 comp} such that
\[
\lim_{x_N \to \pm \infty} \left( u_1(x',x_N) -u_2(x',x_N) \right) = \pm \infty
\]
uniformly in $x' \in \R^{N-1}$, and Corollary 1.2 of \cite{FaSo} gives the desired result. 
\end{proof}

\end{document}